\documentclass[11pt,english]{article}
\usepackage[T1]{fontenc}
\usepackage[latin9]{inputenc}
\usepackage{geometry}
\usepackage{color}
\usepackage{amsmath}
\usepackage{amsthm}
\usepackage{amssymb}

\makeatletter
\theoremstyle{plain}
\newtheorem{thm}{\protect\theoremname}
\theoremstyle{remark}
\newtheorem*{rem*}{\protect\remarkname}
\theoremstyle{remark}
\newtheorem*{acknowledgement*}{\protect\acknowledgementname}
\theoremstyle{definition}
\newtheorem{defn}[thm]{\protect\definitionname}
\theoremstyle{remark}
\newtheorem{note}[thm]{\protect\notename}
\theoremstyle{plain}
\newtheorem{fact}[thm]{\protect\factname}
\theoremstyle{plain}
\newtheorem{lem}[thm]{\protect\lemmaname}
\theoremstyle{remark}
\newtheorem{rem}[thm]{\protect\remarkname}
\theoremstyle{plain}
\newtheorem{prop}[thm]{\protect\propositionname}
\theoremstyle{remark}
\newtheorem{notation}[thm]{\protect\notationname}

\scrollmode
\usepackage{lmodern}
\usepackage{hyperref}
\date {}

\makeatother

\usepackage{babel}
\providecommand{\acknowledgementname}{Acknowledgement}
\providecommand{\definitionname}{Definition}
\providecommand{\factname}{Fact}
\providecommand{\lemmaname}{Lemma}
\providecommand{\notationname}{Notation}
\providecommand{\notename}{Note}
\providecommand{\propositionname}{Proposition}
\providecommand{\remarkname}{Remark}
\providecommand{\theoremname}{Theorem}

\begin{document}
\global\long\def\RR{\mathbb{R}}%
\global\long\def\CC{\mathbb{C}}%
\global\long\def\HH{\mathbb{H}}%
\global\long\def\NN{\mathbb{N}}%
\global\long\def\ZZ{\mathbb{Z}}%
\global\long\def\QQ{\mathbb{Q}}%
\global\long\def\FF{\mathbb{F}}%
\global\long\def\KK{\mathbb{K}}%
\global\long\def\th{\theta}%
\global\long\def\e{\epsilon}%
\global\long\def\a{\alpha}%
\global\long\def\b{\beta}%
\global\long\def\ga{\gamma}%
\global\long\def\ph{\varphi}%
\global\long\def\om{\omega}%
\global\long\def\lm{\lambda}%
\global\long\def\dl{\delta}%
\global\long\def\little{\varepsilon}%
\global\long\def\gam{\Gamma}%
\global\long\def\lra{\longrightarrow}%
\global\long\def\del{\partial}%
\global\long\def\half{\frac{1}{2}}%
\global\long\def\exd#1#2{\underset{#2}{\underbrace{#1}}}%
\global\long\def\exup#1#2{\overset{#2}{\overbrace{#1}}}%
\global\long\def\diffeo{\simeq}%
\global\long\def\mtx{T}%
\global\long\def\algrp{\mathbf{G}}%
\global\long\def\transpose{\mbox{t}}%
\global\long\def\norm#1{\left\Vert #1\right\Vert }%
\global\long\def\normp#1{\left\Vert #1\right\Vert _{p}}%
\global\long\def\abs#1{\left|#1\right|}%
\global\long\def\absp#1{\left|#1\right|_{p}}%
\global\long\def\dir#1{\hat{#1}}%
\global\long\def\dirp#1{\check{#1}}%
\global\long\def\measp{\mu_{p}}%
\global\long\def\id{\operatorname{id}}%
\global\long\def\idmat#1{\operatorname{I}_{#1}}%
\global\long\def\gl#1{\operatorname{GL}_{#1}}%
\global\long\def\sl#1{\operatorname{SL}_{#1}}%
\global\long\def\so#1{\operatorname{SO}_{#1}}%
\global\long\def\ort#1{\operatorname{O}_{#1}}%
\global\long\def\pgl#1{\operatorname{PGL}_{#1}}%
\global\long\def\po#1{\operatorname{PO}_{#1}}%
\global\long\def\leb{\operatorname{Leb}}%
\global\long\def\diag{\operatorname{diag}}%
\global\long\def\prim{\operatorname{prim}}%
\global\long\def\Ad#1{\operatorname{Ad}_{#1}}%
\global\long\def\sphere#1{\mathbb{S}^{#1}}%
\global\long\def\arc{\Theta}%
\global\long\def\arcp{\Theta_{p}}%
\global\long\def\nbhd#1#2{\mathcal{O}_{#1}^{#2}}%
\global\long\def\ball#1{B_{#1}}%
\global\long\def\BIcomp{H}%
\global\long\def\Gset{\mathcal{B}}%
\global\long\def\Mat#1{\text{Mat}_{#1}}%
\global\long\def\Nset{\Psi}%
\global\long\def\val{\nu}%
\global\long\def\intdom{\mathcal{O}}%
\global\long\def\leftgrp{Q}%
\global\long\def\errexp{\tau}%

\title{$p$-adic Directions of Primitive Vectors}
\author{Antonin Guilloux\thanks{IMJ-PRG, OURAGAN, Sorbonne Universit\'e, CNRS, INRIA, \texttt{antonin.guilloux@imj-prg.fr}.}\and
Tal Horesh \thanks{IST Austria, \texttt{tal.horesh@ist.ac.at}. Supported by EPRSC grant EP/P026710/1.}
\and }

\maketitle

\begin{abstract}
Linnik type problems concern the distribution of projections of integral
points on the unit sphere as their norm increases, and different generalizations
of this phenomenon. Our work addresses a question of this type: we
prove the uniform distribution of the projections of primitive $\ZZ^{2}$
points in the $p$-adic unit sphere, as their (real) norm tends to
infinity. The proof is via counting lattice points in semi-simple
$S$-arithmetic groups. 
\end{abstract}

A primitive vector is an $n$-tuple $\left(a_{1},\ldots,a_{n}\right)$
of co-prime integers, and we let $\ZZ_{\prim}^{n}$ denote the set
of primitive vectors in $\ZZ^{n}$. Since every integral vector is
an integer multiple of a unique primitive vector, it is very natural
to restrict questions about equidistribution of integer vectors to
the set of primitive vectors. For example, one question about an equidistribution
property for integer vectors that has been studied in the past (e.g.\
in \cite{Schmidt_98}) is whether the directions of integral vectors,
i.e. their projections to the unit sphere in $\RR^{n}$, distribute
uniformly in the unit sphere as their norm tends to $\infty$. This
question belongs to the well known family of \emph{Linnik type problems}
(e.g. \cite{Linnik_68,Duke_03,Duke_07}), and the answer is positive:
for every ``reasonable'' subset $\arc$ of the sphere, it holds
that 
\begin{equation}
\frac{\#\left\{ \frac{v}{\norm v}\in\arc:v\in\ZZ^{n},\norm v\leq R\right\} }{\#\left\{ \frac{v}{\norm v}\in\sphere{n-1}:v\in\ZZ^{n},\norm v\leq R\right\} }\underset{R\to\infty}{\lra}\leb(\arc),\label{eq: equidistribution}
\end{equation}
where $\leb$ is the Lebesgue measure on the sphere. While in the
above quotient every ``integral direction'' on the unit sphere is
hit several times (the first time for a primitive vector, and then
another time for each one of its integer multiples), restricting to
$v\in\ZZ_{\prim}^{n}$ allows every integral direction to be considered
exactly once.

Questions about equidistribution of directions, as well as of other
parameters of primitive vectors, have been studied recently using
dynamical methods in \cite{Marklof_10,AES_16A,AES_16B,EMSS_16,ERW17}.
In the present paper we restrict to dimension $n=2$, and study the
equidistribution of \emph{$p$-adic directions} of primitive vectors.
Indeed, since primitive vectors have integer coordinates, they can
be seen as vectors over any field that contains the rationals, and
in particular over the field of $p$-adic numbers $\QQ_{p}$ for a
positive prime number $p$. There, just like the direction of a real
vector is its projection to the (real) unit sphere through multiplication
by inverse of the norm, the $p$-adic direction of a vector is its
projection to the $p$-adic unit sphere. However, the primitive vectors
have $p$-adic norm one, so in fact they are already contained in
the $p$-adic unit sphere $\sphere 1_{p}$ (we will observe this below,
where we recall some basic definitions in the $p$-adic setting).
So, $\ZZ_{\prim}^{2}$ is a countable subset of $\sphere 1_{p}$ which
is equipped with a natural height function: the real norm. One is
then led to ask whether the set $\ZZ_{\prim}^{2}$ equidistributes
in $\sphere 1_{p}$, i.e., if an analog to (\ref{eq: equidistribution})
holds when $\ZZ^{n}$ is replaced by $\ZZ_{\prim}^{2}$, and $\sphere 1$
is replaced by $\sphere 1_{p}$. To formulate such an analog, we need
to declare what are the analogous objects in $\sphere 1_{p}$ for
an arc on the (real) unit circle, and for the Lebesgue measure on
it. Below, we will define the concept of a $p$-adic arc, $\arcp\subset\sphere 1_{p}$,
and recall a Haar measure $\mu_{p}$ on $\QQ_{p}$. Since $\sphere 1_{p}$
is an open and compact subset of the $p$-adic plane, then the restriction
of the Haar measure $\mu_{p}^{2}$ on the plane $\QQ_{p}^{2}$ to
$\sphere 1_{p}$ is a finite non-zero measure on $\sphere 1_{p}$.
As we shall see below, it corresponds to the Lebesgue measure on the
real unit circle. 

The theorem below establishes the uniform distribution of the primitive
vectors in $\sphere 1_{p}$ as their real norm tends to infinity;
even more, it establishes \emph{joint equidistribution} of their real
and $p$-adic directions in the product of unit circles $\sphere 1\times\sphere 1_{p}$. 
\begin{thm}
\label{thm: A ed of directions}For $v\in\ZZ_{\prim}^{2}$, the pairs
of real and $p$-adic directions
\[
\left(\frac{v}{\norm v},v\right)\in\sphere 1\times\sphere 1_{p}
\]
become uniformly distributed in $\sphere 1\times\sphere 1_{p}$ w.r.t.\
$\leb\times\measp^{2}|_{\sphere 1_{p}}$ as $\norm v\to\infty$, meaning
that for every product of arcs $\arc\times\arcp\subset\sphere 1\times\sphere 1_{p}$
it holds that 
\[
\frac{\#\left\{ v\in\ZZ_{\prim}^{2}:\,\left(\frac{v}{\norm v},v\right)\in\arc\times\arcp,\quad\norm v\leq R\right\} }{\#\left\{ v\in\ZZ_{\prim}^{2}:\quad\norm v\leq R\right\} }\underset{R\to\infty}{\lra}\leb(\arc)\cdot\measp^{2}(\arcp).
\]
The convergence is at rate at most $O\left(R^{-2\errexp_{p}+\dl}\right)$
for every $\dl>0$, where $\errexp_{p}=\frac{1}{28}$. 
\end{thm}

\begin{rem*}
The parameter $\errexp_{p}$ in error term exponent depends on the
rate of decay of the matrix coefficients of automorphic representations
of $\sl 2\left(\RR\right)\times\sl 2\left(\QQ_{p}\right)$, and can
be improved to $\frac{1}{14}$ when assuming the Ramanujan conjecture
(cf.\ Remark \ref{rem: err exponent}).
\end{rem*}
A similar theorem holds when considering a finite set of primes instead
of a single $p$. Theorem \ref{thm: Equidistribution Bruhat-Iwasawa}
in Section \ref{sec:Equidistribution-of-Iwasawa} states a related
equidistribution result in this $S$-arithmetic setting, that translates
into a counting statement as above, see Remark \ref{rem: S-arithmetic prim vectors}.
However, we have chosen to deal with only one prime here in order
to ease the exposition.

\paragraph*{Organization of the paper. }

The first section of the paper is a collection of general facts on
$p$-adic numbers and arithmetic lattices. Sections \ref{sec:Counting-lattice-points}
and \ref{sec:Counting-lattice-points-1} are devoted to the proof
of Theorem \ref{thm: A ed of directions}, along the lines of the
proof given in \cite{HN_Counting,HK_gcd} for the uniform distribution
of the real directions of primitive vectors in the unit sphere. It
consists of two stages: the first is a translation of the theorem
to a statement about counting lattice points in the group $\sl 2\left(\RR\right)\times\sl 2\left(\QQ_{p}\right)$
(Section \ref{sec:Counting-lattice-points}), and the second is proving
the counting statement via a method developed in \cite{GN12} (Section
\ref{sec:Counting-lattice-points-1}). In Section \ref{sec:Equidistribution-of-Iwasawa},
we show that the same method gives a stronger statement than the one
in Theorem \ref{thm: A ed of directions}, namely the joint equidistribution
of real and $p$-adic directions for any finite number of primes $p$
--- cf.\ Theorem \ref{thm: Equidistribution Bruhat-Iwasawa} (and
Remark \ref{rem: S-arithmetic prim vectors}). Section \ref{sec: Proof of p-adic LWR}
is devoted to proving a technical result on well-roundedness in the
$p$-adic setting that is used in Section \ref{sec:Equidistribution-of-Iwasawa}.
\begin{acknowledgement*}
The authors are grateful to Nicolas Bergeron and Fr\'ed\'eric Paulin for
helpful discussions, and to Tim Browning for his valuable comments
on a preliminary version of the preprint.
\end{acknowledgement*}

\section{$p$-adic numbers and arithmetic lattices}

Let us now recall some basic facts on the $p$-adic numbers and arithmetic
lattices.

\subsection{$p$-adic numbers and vector spaces}
\begin{defn}[\textbf{$p$-adic valuation and absolute value}]
For a non-zero $p$-adic number $a$, the \emph{$p$-adic valuation}
of $a$ is defined to be the biggest integer $\val\left(a\right)$
such that 
\[
a=\sum_{i=\val\left(a\right)}^{\infty}\a_{i}p^{i},\quad\a_{i}\in\left\{ 0,1,\ldots,p-1\right\} ,
\]
where for $a=0$ one defines $\val\left(0\right)=\infty$. The \emph{$p$-adic
absolute value} is given by 
\[
\absp a=p^{-\val\left(a\right)}.
\]
The ring of \emph{$p$-adic integers} $\ZZ_{p}<\QQ_{p}$ is the $p$-adic
unit ball, namely the set of $p$-adic numbers with absolute value
at most $1$ (equivalently, of non-negative valuation). Inside $\ZZ_{p}$,
the set of \emph{invertible} $p$-adic integers $\ZZ_{p}^{*}\subset\ZZ_{p}$
is the set of $p$-adic numbers with absolute value $1$ (equivalently,
of valuation $0$).
\end{defn}

\begin{note}
\label{note: representation of p-adic number }It is easy to see that
every $a\in\QQ_{p}$ can be written (uniquely) as 
\[
a=u_{a}p^{\val\left(a\right)}=u_{a}\left|a\right|_{p}^{-1}
\]
where $u_{a}\in\ZZ_{p}^{\times}$.
\end{note}

A norm on $\QQ_{p}^{2}$ (and therefore a unit circle) is then defined
as follows. 
\begin{defn}[\textbf{norm and unit circle in $\QQ_{p}^{2}$}]
\label{def: p-adic norm}The $p$-adic norm (or just ``norm'')
of a vector $(a,b)\in\QQ_{p}^{2}$ is defined to be
\[
\normp{(a,b)}=\max\left\{ \absp a,\absp b\right\} =p^{-\min\left\{ \val\left(a\right),\val\left(b\right)\right\} }.
\]
Accordingly, the \emph{$p$-adic unit circle} in $\QQ_{p}^{2}$ is
the set of vectors of norm one:
\[
\sphere 1_{p}:=\left\{ v\in\QQ_{p}^{2}:\normp v=1\right\} .
\]
\end{defn}

Note  that $(a,b)$ is in $\sphere 1_{p}$ if and only if both $a,b$
are in $\ZZ_{p}$, and at least one of them is in $\ZZ_{p}^{*}$.
In particular, since the ``usual'' integers are also $p$-adic integers,
and since an integer lies in $\ZZ_{p}^{*}$ if and only if it is
not divisible by $p$, we have that $\ZZ_{\prim}^{2}\subset\sphere 1_{p}$.

Since $\ZZ_{p}$ is the unit ball in $\QQ_{p}$, then a ball of radius
$p^{-N}$ around $\a\in\QQ_{p}$ is $\a+p^{N}\ZZ_{p}$. Similarly,
a ball of radius $p^{-N}$ around $(\a,\b)\in\QQ_{p}^{2}$ is $(\a,\b)+p^{N}\ZZ_{p}^{2}$.
The analog in $\QQ_{p}^{2}$ for an arc $\arc$ in the unit circle
$\sphere 1$ is a ball that is contained in the $p$-adic circle:
\begin{defn}[\textbf{$p$-adic arc}]
 Let $N>0$ be an integer. A \emph{$p$-adic arc} of radius $p^{-N}$
is a ball $\th+p^{N}\ZZ_{p}^{2}$, where $\th\in\sphere 1_{p}$. 
It will be denoted by $\arcp=\arc_{p}\left(\th,p^{-N}\right)$. 
\end{defn}

Note that it is sufficient that $\th\in\sphere 1_{p}$ in order to
have $\arcp\subset\sphere 1_{p}$!

A Haar measure $\mu_{p}$ is defined on $\QQ_{p}$ (see, e.g.\ \cite{BS05})
by assigning to a ball of radius $p^{-N}$ the volume $p^{-N}$. We
let $\measp^{2}:=\measp\times\measp$ denote the resulting Haar measure
on $\QQ_{p}^{2}$, which then assigns to a ball of radius $p^{-N}$
the volume $p^{-2N}$. Being a Haar measure on $\QQ_{p}^{2}$, $\measp^{2}$
is invariant under the group $\sl 2(\QQ_{p})$ of $p$-adic $2$ by
$2$ matrices with determinant one. We note that unlike the real
case, in the $p$-adic plane the unit circle has positive Haar measure.
Indeed, it contains the subset $\ZZ_{p}^{\times}\times\ZZ_{p}$, which
has measure 
\[
\measp(\ZZ_{p}^{\times})\cdot\measp(\ZZ_{p})=\measp(\ZZ_{p}-p\ZZ_{p})\cdot\measp(\ZZ_{p})=(1-\frac{1}{p})\cdot1=1-\frac{1}{p}.
\]
Hence, it is possible to restrict $\measp^{2}$ to $\sphere 1_{p}$.
This measure is the analog of the Lebesgue measure on the real unit
circle, in the sense that it is invariant under the group of norm
preserving linear transformations of $\QQ_{p}^{2}$. Indeed, it is
known (e.g. \cite[Cor. 3.3]{GI63}) that the group $\sl 2(\ZZ_{p})$
is the stabilizer of the norm $\normp{\cdot}$ on $\QQ_{p}^{2}$,
so in particular it preserves and acts transitively on $\sphere 1_{p}$. 

\subsection{$S$-arithmetic lattices}

The field $\QQ_{p}$ does not contain a lattice; but, inside the ring
\[
\FF:=\RR\times\QQ_{p},
\]
the subring $\ZZ\left[\frac{1}{p}\right]$ of polynomials in $\frac{1}{p}$
with integer coefficients embeds diagonally as a co-compact lattice.
It is an integral domain, and as such it plays the role of the integral
lattice inside $\FF$. According to the Borel Harish--Chandra Theorem
\cite{Borel_HarishChandra}, the diagonal embedding of the subgroup
$\sl 2\left(\ZZ\left[\frac{1}{p}\right]\right)$, which by abuse of
notation we denote by $\sl 2\left(\ZZ\left[\frac{1}{p}\right]\right)$,
is a lattice inside 
\[
\sl 2\left(\FF\right):=\sl 2\left(\RR\right)\times\sl 2\left(\QQ_{p}\right).
\]
Both these lattices, $\ZZ\left[\frac{1}{p}\right]<\FF$ and $\sl 2\left(\ZZ\left[\frac{1}{p}\right]\right)<\sl 2\left(\FF\right)$,
are a special case of \emph{$S$-arithmetic lattices} \cite{PPR_92}.
To familiarize the reader with the $S$-arithmetic framework, we include
a proof that these two discrete subgroups are indeed lattices, by
establishing the existence of finite-volume fundamental domains. Here
a fundamental domain means  a full set of representatives. 

\begin{fact}
\label{fact: fund dom for SL2(Z=00005B1/p=00005D)}Let $\mathcal{D}_{\infty}$
denote a fundamental domain for $\sl 2\left(\ZZ\right)$ in $\sl 2\left(\RR\right)$.
Then
\begin{enumerate}
\item $\left[\half,\half\right)\times\ZZ_{p}$ is a fundamental domain for
$\ZZ\left[\frac{1}{p}\right]$ in $\FF$. 
\item $\mathcal{D}_{\infty}\times\sl 2\left(\ZZ_{p}\right)$ is a fundamental
domain for $\sl 2\left(\ZZ\left[\frac{1}{p}\right]\right)$ in $\sl 2\left(\FF\right)$.
\end{enumerate}
\end{fact}

\begin{proof}
1. Given $\left(x,\a\right)\in\RR\times\QQ_{p}$, with $\a=\sum_{i=\val\left(\a\right)}^{\infty}\a_{i}p^{i}$,
write 
\[
\left\{ \a\right\} =\sum_{i=\val\left(\a\right)}^{-1}\a_{i}p^{i}\in\ZZ\left[\frac{1}{p}\right].
\]
Then clearly $\a-\left\{ \a\right\} =\sum_{i=0}^{\infty}\a_{i}p^{i}\in\ZZ_{p}$,
and so 
\[
\left(x,\a\right)-\left(\left\{ \a\right\} ,\left\{ \a\right\} \right)\in\RR\times\ZZ_{p}.
\]
Now let $m\in\ZZ$ be an integer such that $x-\left\{ \a\right\} -m\in\left[\half,\half\right)$.
Since $\ZZ\subset\ZZ_{p}$, then 
\[
\left(x,\a\right)-\left(\left\{ \a\right\} +m,\left\{ \a\right\} +m\right)\in\left[\half,\half\right)\times\ZZ_{p}.
\]
For uniqueness, assume $\left(x,\a\right)$ and $\left(y,\b\right)$
are both in $\left[\half,\half\right)\times\ZZ_{p}$, with $f\in\ZZ\left[\frac{1}{p}\right]$
such that 
\[
\left(x,\a\right)+\left(f,f\right)=\left(y,\b\right).
\]
Then $f=\b-\a\in\ZZ_{p}$ must be an integer since $\ZZ\left[\frac{1}{p}\right]\cap\ZZ_{p}=\ZZ$,
and the real coordinate forces that $f=0$. 

2. For the second part, let $g=\left(g_{\infty},g_{p}\right)\in\sl 2\left(\RR\right)\times\sl 2\left(\QQ_{p}\right)$;
we show that there exists a unique $\left(\ga,\ga\right)\in\sl 2\left(\ZZ\left[\frac{1}{p}\right]\right)$
such that 
\[
\left(g_{\infty},g_{p}\right)\cdot\left(\ga,\ga\right)\in\mathcal{D}_{\infty}\times\sl 2\left(\ZZ_{p}\right).
\]
It is a consequence of row reduction that any $g_{p}\in\sl 2\left(\QQ_{p}\right)$
can be written (non-uniquely) as 
\[
g_{p}=k_{p}\ga_{p}
\]
with $k_{p}\in\sl 2\left(\ZZ_{p}\right)$ and $\ga_{p}\in\sl 2\left(\ZZ\left[\frac{1}{p}\right]\right)$.
Then $\left(g_{\infty},g_{p}\right)=\left(g_{\infty},k_{p}\ga_{p}\right)$
meaning that 
\[
\left(g_{\infty},g_{p}\right)\left(\ga_{p}^{-1},\ga_{p}^{-1}\right)=\left(g_{\infty}\ga_{p}^{-1},k_{p}\right).
\]
Write 
\[
\sl 2\left(\RR\right)\ni g_{\infty}\ga_{p}^{-1}=x_{\infty}\ga_{\infty}
\]
where $x_{\infty}\in\mathcal{D}_{\infty}$ and $\ga_{\infty}\in\sl 2\left(\ZZ\right)$.
Then
\[
\left(g_{\infty},g_{p}\right)\left(\ga_{p}^{-1}\ga_{\infty}^{-1},\ga_{p}^{-1}\ga_{\infty}^{-1}\right)=(x_{\infty},k_{p}\ga_{\infty}^{-1})\in\mathcal{D}_{\infty}\times\sl 2\left(\ZZ_{p}\right),
\]
which establishes existence. For uniqueness of $\left(\ga,\ga\right)=(\ga_{p}^{-1}\ga_{\infty}^{-1},\ga_{p}^{-1}\ga_{\infty}^{-1})$,
assume that $\left(g_{\infty},g_{p}\right)\cdot\left(\ga,\ga\right)=\left(x_{p},x_{\infty}\right)$
and $\left(g_{\infty},g_{p}\right)\cdot\left(\ga^{\prime},\ga^{\prime}\right)=\left(x_{\infty}^{\prime},x_{p}^{\prime}\right)$,
where both $\left(x_{\infty},x_{p}\right)$ and $\left(x_{\infty}^{\prime},x_{p}^{\prime}\right)$
lie in $\mathcal{D}_{\infty}\times\sl 2\left(\ZZ_{p}\right)$. Then
\[
\left(x_{\infty},x_{p}\right)\cdot\left(\ga^{-1}\ga^{\prime},\ga^{-1}\ga^{\prime}\right)=\left(x_{\infty}^{\prime},x_{p}^{\prime}\right)\in\mathcal{D}_{\infty}\times\sl 2\left(\ZZ_{p}\right)
\]
where $\ga^{-1}\ga^{\prime}\in\sl 2\left(\ZZ\left[\frac{1}{p}\right]\right)$.
From the $p$-adic component of the equation we have that $\ga^{-1}\ga^{\prime}\in\sl 2\left(\ZZ_{p}\right)$.
Then $\ga^{-1}\ga^{\prime}\in\sl 2\left(\ZZ\left[\frac{1}{p}\right]\right)\cap\sl 2\left(\ZZ_{p}\right)=\sl 2\left(\ZZ\right)$.
On the other hand, from the real component of the equation,  we have
that $\ga^{-1}\ga^{\prime}$ cannot lie in $\sl 2\left(\ZZ\right)$,
unless $\ga^{-1}\ga^{\prime}=\id$. We conclude that $\ga=\ga^{\prime}$. 
\end{proof}

\section{From primitive vectors to lattice points in the group $\protect\sl 2\left(\protect\RR\right)\times\protect\sl 2\left(\protect\QQ_{p}\right)$
\label{sec:Counting-lattice-points}}

It is well known that there exists a connection between primitive
vectors in $\ZZ^{2}$, and integral matrices inside $\sl 2\left(\RR\right)$;
the goal of this section is to establish an analogous connection in
the setting of $\RR\times\QQ_{p}$, hence reducing the proof of Theorem
\ref{thm: A ed of directions} to counting lattice points in the group
$\sl 2\left(\RR\right)\times\sl 2\left(\QQ_{p}\right)$. To exhibit
such a connection, we first extend the notion of primitive vectors
(subsection \ref{subsec:Primitive-vectors-over}), then find suitable
coordinates on $\sl 2\left(\FF\right)$ through a Bruhat-Iwasawa decomposition
(Definition \ref{def:Iwasawa-Bruhat}). Proposition \ref{prop: matrices representing  Z=00005B1/p=00005D primitive vectors }
states the precise connection.

\subsection{Primitive vectors over $\protect\ZZ\left[\frac{1}{p}\right]$\label{subsec:Primitive-vectors-over}}

We aim to formulate a connection between primitive vectors in $\ZZ\left[\frac{1}{p}\right]^{2}$
and matrices inside $\sl 2\left(\ZZ\left[\frac{1}{p}\right]\right)$.
First, let us extend the definition of a primitive vector from $\ZZ$
to a general integral domain: 
\begin{defn}
Let $\intdom$ be an integral domain. A vector $v=\left(a,b\right)\in\intdom^{2}$
is called \emph{primitive} if  the prime ideals $\left\langle a\right\rangle $
and $\left\langle b\right\rangle $ satisfy that $\left\langle a\right\rangle +\left\langle b\right\rangle =\intdom$.
In other words, if there exists a solution $\left(x,y\right)\in\intdom^{2}$
to the $\intdom$--diophantine equation 
\[
ax+by=1.
\]
We refer to this equation as the\emph{ gcd equation} of $v$, and
denote the set of primitive elements in $\intdom^{2}$ by $\intdom_{\prim}^{2}$. 
\end{defn}

Clearly, the set of primitive $\ZZ\left[\frac{1}{p}\right]^{2}$ vectors
contains the set of primitive $\ZZ^{2}$ vectors, but is not equal
to it; e.g.\ the vector $(p,0)$ is primitive in $\ZZ\left[\frac{1}{p}\right]^{2}$
but not in $\ZZ^{2}$. However, these sets are very much related to
each other: every element in $\ZZ\left[\frac{1}{p}\right]_{\prim}^{2}$
is a multiplication by a power of $p$ of an element in $\ZZ_{\prim}^{2}$. 
\begin{lem}
\label{fact: primitive vectors over Z=00005B1/p=00005D}The primitive
vectors in $\ZZ\left[\frac{1}{p}\right]^{2}$ are $\left\{ p^{\a}v:\a\in\ZZ,\,v\in\ZZ_{\prim}^{2}\right\} $. 
\end{lem}

For the proof, we observe that every element $f$ of $\ZZ\left[\frac{1}{p}\right]$
can be written as $f=\frac{m}{p^{n}}$ where $m,n\in\ZZ$ and $m$
is coprime to $p$. If $\deg(f)=0$ (as a polynomial in $\frac{1}{p}$),
then $f$ is an integer and this fact is clear. Otherwise, let $n>0$
and write 
\[
f=\frac{a_{n}}{p^{n}}+\frac{a_{n-1}}{p^{n-1}}+\cdots+\frac{a_{1}}{p}+a_{0}
\]
where $a_{0},\ldots,a_{n}$ are integers that are (except maybe $a_{0}$)
coprime to $p$, and $a_{n}\neq0$. Then 
\[
f=\frac{a_{n}+a_{n-1}p+\cdots+a_{1}p^{n-1}+a_{0}p^{n}}{p^{n}},
\]
and the denominator is coprime to $p$ when $n>0$.
\begin{proof}[Proof of Lemma \ref{fact: primitive vectors over Z=00005B1/p=00005D}]
Let $\left(a,b\right)\in\ZZ\left[\frac{1}{p}\right]_{\prim}^{2}$
and write $a=a^{\prime}p^{\a}$ and $b=b^{\prime}p^{\b}$ where $a^{\prime},b^{\prime}\in\ZZ$
are co-prime to $p$. We claim that $\gcd(a^{\prime},b^{\prime})=1$.
To see this, substitute $a$ and $b$ into the $\gcd$ equation $ax+by=1$
($x,y$ variables in $\ZZ\left[\frac{1}{p}\right]$) to obtain $a^{\prime}p^{\a}x+b^{\prime}p^{\b}y=1$.
Multiply both sides by a non-negative power of $p$ to obtain an equation
in integers: $a^{\prime}x^{\prime}+b^{\prime}y^{\prime}=p^{m}$ where
$x^{\prime},y^{\prime}\in\ZZ$. It then follows that $\gcd\left(a^{\prime},b^{\prime}\right)$
is a power of $p$, but since both $a^{\prime},b^{\prime}$ are co-prime
to $p$, then $\gcd\left(a^{\prime},b^{\prime}\right)=1$. Without
loss of generality, assume $\b\geq\a$. Then $\left(a,b\right)=p^{\a}\left(a^{\prime},b^{\prime}p^{\b-\a}\right)=p^{\a}v$
where $v=\left(a^{\prime},b^{\prime}p^{\b-\a}\right)\in\ZZ_{\prim}^{2}$. 
\end{proof}
To formulate the connection between primitive $\ZZ\left[\frac{1}{p}\right]^{2}$
vectors to matrices in $\sl 2\left(\ZZ\left[\frac{1}{p}\right]\right)$,
we introduce the Bruhat-Iwasawa decomposition of $\sl 2\left(\FF\right)$.

\subsection{Iwasawa decomposition of $\protect\sl 2\left(\protect\RR\right)$ }

Let us first recall the $KAN$ decomposition of $\sl 2\left(\RR\right)$.
Here $\sl 2\left(\RR\right)=K_{\infty}A_{\infty}N_{\infty}$ where
$K_{\infty}=\so 2\left(\RR\right)$ is maximal compact, $A_{\infty}=\left\{ \left(\begin{smallmatrix}\a & 0\\
0 & \a^{-1}
\end{smallmatrix}\right)\right\} $ is the diagonal subgroup and $N_{\infty}=\left\{ \left(\begin{smallmatrix}1 & \RR\\
0 & 1
\end{smallmatrix}\right)\right\} $ is the subgroup of upper unipotent matrices. Then, for $\left(\begin{smallmatrix}a & c\\
b & d
\end{smallmatrix}\right)\in\sl 2\left(\RR\right)$, 
\[
\left(\begin{array}{cc}
a & c\\
b & d
\end{array}\right)=\exd{\frac{1}{\sqrt{a^{2}+b^{2}}}\left(\begin{array}{cc}
a & -b\\
b & a
\end{array}\right)}{\in K_{\infty}}\exd{\left(\begin{array}{cc}
\sqrt{a^{2}+b^{2}} & 0\\
0 & \frac{1}{\sqrt{a^{2}+b^{2}}}
\end{array}\right)}{\in A_{\infty}}\exd{\left(\begin{array}{cc}
1 & \frac{ac+bd}{a^{2}+b^{2}}\\
0 & 1
\end{array}\right)}{\in N_{\infty}}.
\]
By letting $v:=\left(a,b\right)^{\transpose}$ and $w:=\left(c,d\right)^{\transpose}$,
we obtain
\begin{equation}
\left(\begin{array}{cc}
v & w\end{array}\right)=\exd{\left(\begin{array}{cc}
\hat{v} & \hat{v}^{\perp}\end{array}\right)}{\in K_{\infty}}\exd{\left(\begin{array}{cc}
\norm v & 0\\
0 & \norm v^{-1}
\end{array}\right)}{\in A_{\infty}}\exd{\left(\begin{array}{cc}
1 & \frac{\left\langle w,\hat{v}\right\rangle }{\norm v}\\
0 & 1
\end{array}\right)}{\in N_{\infty}},\label{eq: Iwasawa SL(2,R)}
\end{equation}
where 
\[
\hat{v}:=v/\norm v
\]
is the unit vector pointing in the direction of $v$, 
\[
v^{\perp}:=\left(-b,a\right)
\]
is the vector pointing in the orthogonal direction to $v$, and $\left\langle \cdot,\cdot\right\rangle $
is the standard dot product in $\RR^{2}$. 

We note that any pair of the three subgroups $K_{\infty}$, $A_{\infty}$
and $N_{\infty}$ intersect trivially, and therefore the decomposition
$\sl 2\left(\RR\right)=K_{\infty}A_{\infty}N_{\infty}$ induces coordinates
on $\sl 2\left(\RR\right)$: every  element $g\in\sl 2\left(\RR\right)$
has a unique presentation as $g=kan$. In addition, a Haar measure
on $\sl 2\left(\RR\right)$ can be decomposed in the Iwasawa coordinates
as 
\[
d\mu_{\sl 2\left(\RR\right)}(g)=d\mu_{\sl 2\left(\RR\right)}(kan)=\frac{d\mu_{K_{\infty}}(k)d\mu_{A_{\infty}}(a)d\mu_{N_{\infty}}(n)}{\a},
\]
where: $\mu_{N_{\infty}}$ is the Haar measure on $N_{\infty}$ corresponding
to the Lebesgue measure on $\RR$ under the isomorphism $\left(\begin{smallmatrix}1 & x\\
0 & 1
\end{smallmatrix}\right)\leftrightarrow x$; $\mu_{K_{\infty}}$ is the Haar measure on $K_{\infty}$ corresponding
to the Lebesgue measure on $\sphere 1$ under the isomorphism $\left(\begin{smallmatrix}\cos\theta & \sin\theta\\
-\sin\theta & \cos\theta
\end{smallmatrix}\right)\leftrightarrow\theta$; and $\mu_{A_{\infty}}$ is a Haar measure on $A_{\infty}$ corresponding
to the Lebesgue measure on the multiplicative group of $\RR_{>0}$
under the isomorphism $\left(\begin{smallmatrix}\a & 0\\
0 & \a^{-1}
\end{smallmatrix}\right)\leftrightarrow\a$. With these isomorphisms in mind, and recalling that the Haar measure
on $\left(\RR_{>0},\cdot\right)$ is $\frac{d\a}{\a}$, we have
\begin{equation}
d\mu_{\sl 2\left(\RR\right)}(g)=d\mu_{\sl 2\left(\RR\right)}\left(\left(\begin{smallmatrix}\cos\theta & \sin\theta\\
-\sin\theta & \cos\theta
\end{smallmatrix}\right)\left(\begin{smallmatrix}\a & 0\\
0 & \a^{-1}
\end{smallmatrix}\right)\left(\begin{smallmatrix}1 & x\\
0 & 1
\end{smallmatrix}\right)\right)=\frac{d\theta d\a dx}{\a^{2}},\label{eq: Iwasawa measure}
\end{equation}
where $dy$ stands for integration w.r.t.\ the Lebesgue measure
on $\RR$.

\subsection{Bruhat decomposition of $\protect\sl 2\left(\protect\QQ_{p}\right)$}

Proceeding to the $p$-adic case, the group $\sl 2\left(\QQ_{p}\right)$
can also be written as $K_{p}A_{p}N_{p}$, but these are not coordinates.
Indeed, set 
\begin{equation}
\begin{array}{cclcc}
K_{p} & := & \left\{ \left(\begin{array}{cc}
a & c\\
b & d
\end{array}\right):a,b,c,d\in\ZZ_{p},ad-bc=1\right\}  & = & \sl 2\left(\ZZ_{p}\right)\\
A_{p} & := & \left\{ \left(\begin{array}{cc}
p^{-t} & 0\\
0 & p^{t}
\end{array}\right):t\in\ZZ\right\} \\
N_{p} & := & \left\{ \left(\begin{array}{cc}
1 & \a\\
0 & 1
\end{array}\right):\a\in\QQ_{p}\right\} .
\end{array}\label{eq: KAN SL(2,Q_p)}
\end{equation}
Since $K_{p}$ and $N_{p}$ intersect  non-trivially, the Iwasawa
decomposition on $\sl 2\left(\QQ_{p}\right)$ is not unique. To remedy
this, we use the Bruhat decomposition instead. Let 
\[
M_{p}:=\left\{ \left(\begin{array}{cc}
u & 0\\
0 & u^{-1}
\end{array}\right):u\in\ZZ_{p}^{*}\right\} 
\]
 be the centralizer of $A_{p}$ in $K_{p}$, $D_{p}:=\left\{ \left(\begin{smallmatrix}\a & 0\\
0 & \a^{-1}
\end{smallmatrix}\right):\a\in\QQ_{p}\right\} $ be the diagonal subgroup (satisfying $D_{p}=M_{p}A_{p}$), and 
\[
N_{p}^{-}:=\left\{ \left(\begin{array}{cc}
1 & 0\\
\a & 1
\end{array}\right):\a\in\QQ_{p}\right\} 
\]
be the lower unipotent subgroup. Each two of the subgroups $N_{p}^{-}$,
$M_{p}$, $A_{p}$ and $N_{p}$ intersect trivially, which means that
they induce coordinates --- the Bruhat coordinates --- on the subset
\[
N_{p}^{-}M_{p}A_{p}N_{p}\subset\sl 2(\QQ_{p}).
\]
This set is not the whole of $\sl 2(\QQ_{p})$, but its complement
in $\sl 2\left(\QQ_{p}\right)$ is of Haar measure zero. 

We choose the following Haar measures on the above subgroups. The
natural isomorphisms between $N_{p}$ and $N_{p}^{-}$ to $\QQ_{p}$
equip $N_{p}$ and $N_{p}^{-}$ with Haar measures that correspond
to the Haar measure $\mu_{p}$ on $\QQ_{p}$:
\[
d\mu_{N_{p}}\left(\begin{smallmatrix}1 & \b\\
0 & 1
\end{smallmatrix}\right)=d\mu_{p}(\b),\qquad d\mu_{N_{p}^{-}}\left(\begin{smallmatrix}1 & 0\\
\a & 1
\end{smallmatrix}\right)=d\mu_{p}(\a).
\]
Similarly, the natural isomorphism between $A_{p}$ to $\ZZ$ induces
$A_{p}$ with the Haar measure on $\ZZ$ that is the counting measure:
\[
\mu_{A_{p}}=\sum_{a\in A_{p}}\Delta_{a}.
\]
Finally, $D_{p}$ is naturally isomorphic with the multiplicative
group $\QQ_{p}^{\times}$ of the field $\QQ_{p}$, from which it inherits
the Haar measure:
\[
d\mu_{D_{p}}\left(\begin{smallmatrix}\a & 0\\
0 & \a^{-1}
\end{smallmatrix}\right)=d\mu_{\QQ_{p}^{\times}}(\a)=\frac{d\mu_{p}(\a)}{\absp{\a}}.
\]
Since the measure on $\QQ_{p}^{\times}$ restricts to the measure
on $\ZZ_{p}^{\times}$, and the latter is isomorphic to $M_{p}$ in
the same way that $\QQ_{p}^{\times}$ is isomorphic to $D_{p}$, we
obtain a Haar measure on $M_{p}$:
\[
d\mu_{M_{p}}\left(\begin{smallmatrix}u & 0\\
0 & u^{-1}
\end{smallmatrix}\right)=d\mu_{\ZZ_{p}^{\times}}(u)=\frac{d\mu_{p}(u)}{\absp u}=d\mu_{p}(u).
\]

A Haar measure on $\sl 2\left(\QQ_{p}\right)$ can be expressed in
the Bruhat coordinates as follows.
\begin{lem}
\label{fact: Haar measure on SL2(Q_p)}The Haar measure on $G_{p}=\sl 2\left(\QQ_{p}\right)$
w.r.t. the Bruhat coordinates is 
\[
d\mu_{G_{p}}(g)=d\mu_{G_{p}}\left(\left(\begin{smallmatrix}1 & 0\\
\a & 1
\end{smallmatrix}\right)\left(\begin{smallmatrix}u & 0\\
0 & u^{-1}
\end{smallmatrix}\right)\left(\begin{smallmatrix}p^{-t} & 0\\
0 & p^{t}
\end{smallmatrix}\right)\left(\begin{smallmatrix}1 & \b\\
0 & 1
\end{smallmatrix}\right)\right)=p^{2t}d\mu_{p}(\a)d\mu_{p}(u)\Delta_{t}d\mu_{p}(\b)
\]
\end{lem}

\begin{proof}
Let $P_{p}^{-}$ denote the subgroup of lower triangular matrices
in $\sl 2\left(\QQ_{p}\right)$, meaning that $P_{p}^{-}=N_{p}^{-}D_{p}=N_{p}^{-}M_{p}A_{p}$.
By \cite[Lemma 11.31]{EW13}, since $G_{p}$ is unimodular, a Haar
measure $\mu_{G_{p}}=\mu_{\sl 2\left(\QQ_{p}\right)}$ is given by
 $\mu_{G_{p}}=\mu_{P_{p}^{-}}^{L}\times\mu_{N_{p}}^{R}$, where $\mu_{P_{p}^{-}}^{L}$
is a left Haar measure on $P_{p}^{-}$ and $\mu_{N_{p}}^{R}$ is a
right Haar measure on $N_{p}$, which is simply $\mu_{N_{p}}$. It
is left to compute a left Haar measure on $P_{p}^{-}$. Note that
$P_{p}^{-}$ can be introduced in two ways as $N_{p}^{-}D_{p}=D_{p}N_{p}^{-}$,
but the expression for the Haar measure will correspond to the choice
of coordinates. It is easy to show that $\mu_{D_{p}}\times\mu_{N_{p}^{-}}$
is a left Haar measure on $P_{p}^{-}=D_{p}N_{p}^{-}$; to obtain this
left Haar measure in the coordinates $N_{p}^{-}D_{p}$ we will perform
a change of variables:
\begin{align*}
 & \int\int f\left(\left[\begin{smallmatrix}\a^{-1} & 0\\
o & \a
\end{smallmatrix}\right]\left[\begin{smallmatrix}1 & 0\\
y & 1
\end{smallmatrix}\right]\right)d\mu_{D_{p}}\left(\left[\begin{smallmatrix}\a^{-1} & 0\\
o & \a
\end{smallmatrix}\right]\right)d\mu_{N_{p}^{-}}\left(\left[\begin{smallmatrix}1 & 0\\
y & 1
\end{smallmatrix}\right]\right)\\
 & =\int\int f\left(\left[\begin{smallmatrix}1 & 0\\
y\a^{-2} & 1
\end{smallmatrix}\right]\left[\begin{smallmatrix}\a & 0\\
o & \a^{-1}
\end{smallmatrix}\right]\right)d\mu_{D_{p}}\left(\left[\begin{smallmatrix}\a^{-1} & 0\\
o & \a
\end{smallmatrix}\right]\right)d\mu_{N_{p}^{-}}\left(\left[\begin{smallmatrix}1 & 0\\
y & 1
\end{smallmatrix}\right]\right)\\
 & =\int\int f\left(\left[\begin{smallmatrix}1 & 0\\
x & 1
\end{smallmatrix}\right]\left[\begin{smallmatrix}\a & 0\\
o & \a^{-1}
\end{smallmatrix}\right]\right)\absp{\a}^{2}d\mu_{D_{p}}\left(\left[\begin{smallmatrix}\a^{-1} & 0\\
o & \a
\end{smallmatrix}\right]\right)d\mu_{N_{p}^{-}}\left(\left[\begin{smallmatrix}1 & 0\\
x & 1
\end{smallmatrix}\right]\right)
\end{align*}
Since the integral on the left-hand side is invariant under replacing
$f(g)$ by $f(hg)$ for any $h\in P_{p}^{-}$, then so is the integral
on the right-hand side. Hence $\mu_{P_{p}^{-}}^{L}=\absp{\a}^{2}\left(\mu_{D_{p}}\times\mu_{N_{p}^{-}}\right)$.

Now, since $D_{p}=M_{p}A_{p}$ where $M_{p}$ and $A_{p}$ commute
and are abelian, we have that $\mu_{D_{p}}=\mu_{M_{p}}\times\mu_{A_{p}}$.
We conclude that a left Haar measure on $G_{p}$ is given in the Bruhat
coordinates as 
\[
\mu_{G_{p}}=\absp{\a}^{2}\left(\mu_{N_{p}^{-}}\times\mu_{M_{p}}\times\mu_{A_{p}}\times\mu_{N_{p}}\right).\qedhere
\]
\end{proof}
\begin{rem}
Under this choice of Haar measure, the (compact and open) subgroup
$\sl 2\left(\ZZ_{p}\right)$ has mass $1-\frac{1}{p}$. Indeed, we
have
\[
\sl 2\left(\ZZ_{p}\right)\simeq N_{p}^{-}\left(\ZZ_{p}\right)\times M_{p}\times N_{p}\left(\ZZ_{p}\right)\simeq\ZZ_{p}\times\ZZ_{p}^{\times}\times\ZZ_{p}
\]
and therefore 
\[
\mu_{G_{p}}\left(\sl 2\left(\ZZ_{p}\right)\right)=\measp\left(\ZZ_{p}\right)\measp\left(\ZZ_{p}^{\times}\right)\measp\left(\ZZ_{p}\right)=1\cdot\left(1-\frac{1}{p}\right)\cdot1=1-\frac{1}{p}.
\]
\end{rem}

\subsection{The Bruhat-Iwasawa coordinates }

While the Bruhat decomposition of $\sl 2(\QQ_{p})$ provides uniqueness,
the Iwasawa decomposition provides an arithmetic interpretation of
the different components, as suggested in (\ref{eq: Iwasawa SL(2,R)})
for the $\sl 2(\RR)$ case. Luckily, these two decompositions coincide
on a ``large'' subset of $\sl 2(\QQ_{p})$, on which we are going
to focus from now on. 

Consider $g=\left(\begin{smallmatrix}a & c\\
b & d
\end{smallmatrix}\right)\in\sl 2\left(\QQ_{p}\right)$; assuming that $a\neq0$ we have that 
\begin{align*}
\left(\begin{array}{cc}
a & c\\
b & d
\end{array}\right) & \overset{_{a\neq0}}{=}\exd{\left(\begin{array}{cc}
1 & 0\\
\frac{b}{a} & 1
\end{array}\right)}{\in N_{p}^{-}}\exd{\left(\begin{array}{cc}
a & 0\\
0 & a^{-1}
\end{array}\right)}{\in D_{p}}\exd{\left(\begin{array}{cc}
1 & \frac{c}{a}\\
0 & 1
\end{array}\right)}{\in N_{p}}\\
 & \overset{_{a=u_{a}\absp a^{-1}}}{=}\exd{\left(\begin{array}{cc}
1 & 0\\
\frac{b}{a} & 1
\end{array}\right)}{\in N_{p}^{-}}\exd{\left(\begin{array}{cc}
u_{a} & 0\\
0 & u_{a}^{-1}
\end{array}\right)}{\in M_{p}}\exd{\left(\begin{array}{cc}
\absp a^{-1} & 0\\
0 & \absp a
\end{array}\right)}{\in A_{p}}\exd{\left(\begin{array}{cc}
1 & \frac{c}{a}\\
0 & 1
\end{array}\right)}{\in N_{p}}.
\end{align*}
Indeed, the set of $g\in\sl 2(\QQ_{p})$ with $a\neq0$ is exactly
$N_{p}^{-}M_{p}A_{p}N_{p}$. If we assume further that $\absp a\geq\absp b$,
then $\frac{b}{a}\in\ZZ_{p}$ and therefore the $N_{p}^{-}$ component
lies also in $K_{p}$; letting 
\[
G_{p}^{+}=\sl 2\left(\QQ_{p}\right)^{+}:=\left\{ \left(\begin{array}{cc}
a & c\\
b & d
\end{array}\right)\in\sl 2\left(\QQ_{p}\right):\absp a\geq\absp b\right\} ,
\]
in which necessarily $a\neq0$, we conclude that the Bruhat decomposition
of $G_{p}^{+}$, which is unique, coincides with the Iwasawa decomposition.
This is due to the fact that the $N_{p}^{-}M_{p}$ component in the
Bruhat decomposition  coincides with the $K_{p}$ component in the
Iwasawa decomposition. Denoting this component by 
\[
\leftgrp_{p}:=\left\{ \substack{\text{lower triangular}\\
\text{matrices in \ensuremath{K_{p}}}
}
\right\} =\left\{ \left(\begin{array}{cc}
u & 0\\
m & u^{-1}
\end{array}\right):m\in\ZZ_{p},u\in\ZZ_{p}^{\times}\right\} <K_{p}
\]
(note that it indeed lies in $G_{p}^{+}$), we have that $G_{p}^{+}=\leftgrp_{p}A_{p}N_{p}$,
and these are coordinates on $G_{p}^{+}$. 

Moving forward to the arithmetic interpretation of these coordinates,
it is clear that the first columns of the elements in $G_{p}^{+}$
lie in the $p$-adic ``right half plane'': 
\[
\QQ_{p}^{2\,+}:=\left\{ (a,b)\in\QQ_{p}^{2}:\absp a\geq\absp b\right\} .
\]
Accordingly, the right half of the $p$-adic unit sphere is denoted
\[
\sphere{1,+}_{p}:=\sphere 1_{p}\cap\QQ_{p}^{2\,+}=\left\{ (u,a):u\in\ZZ_{p}^{\times},a\in\ZZ_{p}\right\} .
\]

\begin{fact}
\label{fact: p-adic half-sphere is Q}The half-sphere $\sphere{1,+}_{p}$
is homeomorphic to $\leftgrp_{p}$ and to $\ZZ_{p}^{\times}\times\ZZ_{p}$.
Its Haar measure $\measp^{2}\left(\sphere{1,+}_{p}\right)$ equals
$\measp(\ZZ_{p}^{\times})\measp(\ZZ_{p})=1-\frac{1}{p}=\mu_{G_{p}}\left(\sl 2\left(\ZZ_{p}\right)\right)$.
\end{fact}

To conclude, we have that for $\left(\begin{smallmatrix}a & c\\
b & d
\end{smallmatrix}\right)\in G_{p}^{+}$ (i.e. $v=(a,b)^{\transpose}\in\QQ_{p}^{2\,+}$),
\[
\left(\begin{array}{cc}
a & c\\
b & d
\end{array}\right)\overset{a=u_{a}\absp a^{-1}}{=}\exd{\left(\begin{array}{cc}
u_{a} & 0\\
\frac{b}{a}u_{a} & u_{a}^{-1}
\end{array}\right)}{\in\leftgrp_{p}}\exd{\left(\begin{array}{cc}
\absp a^{-1} & 0\\
0 & \absp a
\end{array}\right)}{\in A_{p}}\exd{\left(\begin{array}{cc}
1 & \frac{c}{a}\\
0 & 1
\end{array}\right)}{\in N_{p}},
\]
which means that the $p$-adic analog to (\ref{eq: Iwasawa SL(2,R)})
is the following. For $g=\left(\begin{array}{cc}
v & w\end{array}\right)\in G_{p}^{+}$, it holds that 
\begin{equation}
\left(\begin{array}{cc}
v & w\end{array}\right)=\exd{\left(\begin{matrix}\dirp v & *\end{matrix}\right)}{\in\leftgrp_{p}}\exd{\left(\begin{array}{cc}
\normp v^{-1} & 0\\
0 & \normp v
\end{array}\right)}{\in A_{p}}\exd{\left(\begin{array}{cc}
1 & \frac{y\left(w\right)}{y\left(v\right)}\\
0 & 1
\end{array}\right)}{\in N_{p}},\label{eq: Bruhat-Iwasawa SL2(Qp)}
\end{equation}
where $v=(x(v),y(v))^{\transpose}$, $w=(x(w),y(w))^{\transpose}$,
and 
\[
\dirp v:=\normp vv
\]
is the unit vector pointing in the $p$-adic direction of $v$ (the
projection of $v$ to the $p$-adic unit sphere).
\begin{defn}
\label{def:Iwasawa-Bruhat}The Iwasawa-Bruhat decomposition of 
\[
\sl 2\left(\FF\right)^{+}:=\sl 2\left(\RR\right)\times\sl 2\left(\QQ_{p}\right)^{+}
\]
is 
\[
\sl 2\left(\FF\right)^{+}=K_{\infty}A_{\infty}N_{\infty}\times\leftgrp_{p}A_{p}N_{p}=\exd{\left(K_{\infty}\times\leftgrp_{p}\right)}{:=\leftgrp}\exd{\left(A_{\infty}\times A_{p}\right)}{:=A}\exd{\left(N_{\infty}\times N_{p}\right)}{:=N}.
\]
\end{defn}

\subsection{Correspondence between primitive vectors in $\protect\ZZ\left[\frac{1}{p}\right]^{2}$
and matrices in $\protect\sl 2\left(\protect\ZZ\left[\frac{1}{p}\right]\right)$}

We now define natural subsets in the Bruhat-Iwasawa components. 
\begin{itemize}
\item For $\mathcal{D}\subset\FF$, we consider 
\[
N_{\mathcal{D}}=\left\{ \left(\left(\begin{smallmatrix}1 & \a\\
0 & 1
\end{smallmatrix}\right),\left(\begin{smallmatrix}1 & a\\
0 & 1
\end{smallmatrix}\right)\right)\in N_{\infty}\times N_{p}:\left(\a,a\right)\in\mathcal{D}\right\} ;
\]
\item for $R>1$, $t_{1}\leq t_{2}\in\ZZ$, let
\[
A_{R,t_{1},t_{2}}=\left\{ \left(\left(\begin{smallmatrix}\a & 0\\
0 & \a^{-1}
\end{smallmatrix}\right),\left(\begin{smallmatrix}p^{-t} & 0\\
0 & p^{t}
\end{smallmatrix}\right)\right)\in A_{\infty}\times A_{p}:\begin{array}{c}
1<\a\leq R\\
t_{1}\leq t\leq t_{2}
\end{array}\right\} ;
\]
\item for a real arc $\arc\subset\sphere 1$ and a $p$-adic arc $\arcp\subset\sphere 1_{p}$,
let
\[
\leftgrp_{\arc,\arcp}=\left\{ \left(\left(\begin{array}{cc}
\dir u & \dir u^{\perp}\end{array}\right),\left(\begin{array}{cc}
\dirp v & *\end{array}\right)\right)\in K_{\infty}\times\leftgrp:\begin{array}{c}
\dir u\in\arc\subset\sphere 1\\
\dirp v\in\arcp\subset\sphere 1_{p}
\end{array}\right\} .
\]
\end{itemize}
The following proposition establishes a 1-to-1 correspondence between
vectors in $\ZZ\left[\frac{1}{p}\right]_{\prim}^{2}$ and certain
matrices in $\sl 2\left(\ZZ\left[\frac{1}{p}\right]\right)$.
\begin{prop}
\label{prop: matrices representing  Z=00005B1/p=00005D primitive vectors }Let
$\mathcal{D}\subset\FF$ be a fundamental domain for the lattice $\ZZ\left[\frac{1}{p}\right]$
in $\FF$. 
\begin{enumerate}
\item \label{enu: 1 bijection}There is a bijection $v\leftrightarrow\ga_{v,D}$
between primitive $\ZZ\left[\frac{1}{p}\right]^{2,+}$ vectors and
$\sl 2\left(\ZZ\left[\frac{1}{p}\right]\right)^{+}$ matrices in $\leftgrp AN_{\mathcal{D}}$.
\item \label{enu: bijection with parameters}For $\mathcal{D}\subset\FF$,
$R>1$, $t_{1}\leq t_{2}\in\ZZ$ and arcs $\arc\subset\sphere 1$,
$\arcp\subset\sphere 1_{p}$, the bijection $v\leftrightarrow\ga_{v,D}$
restricts to being between 
\begin{enumerate}
\item primitive $\ZZ\left[\frac{1}{p}\right]^{2,+}$ vectors of real norm
$\norm v\leq R$, $p$-adic norm $p^{t_{1}}\leq\normp v\leq p^{t_{2}}$
and directions $\left(\dir v,\dirp v\right)\in\arc\times\arcp\subseteq\sphere 1\times\sphere 1_{p}$
\item $\sl 2\left(\ZZ\left[\frac{1}{p}\right]\right)^{+}$ matrices in $\leftgrp_{\arc\times\arcp}A_{R,t_{1},t_{2}}N_{\mathcal{D}}$.
\end{enumerate}
\end{enumerate}
\end{prop}

\begin{proof}[Proof of Proposition \ref{prop: matrices representing  Z=00005B1/p=00005D primitive vectors }]
If $a,b\in\ZZ\left[\frac{1}{p}\right]$ are such that $v=\left(a,b\right)$
is primitive, then there exist (infinitely many) solutions $\left(x,y\right)$
to the $\gcd$ equation of $v$ over $\ZZ\left[\frac{1}{p}\right]$,
$ax+by=1$. If $\left(x_{0},y_{0}\right)$ is such a solution, then
the set of all solutions is 
\[
\left\{ \left(x_{0},y_{0}\right)+m\left(-b,a\right):m\in\ZZ\left[\frac{1}{p}\right]\right\} .
\]
A choice of $m\in\ZZ\left[\frac{1}{p}\right]$ sets a unique solution
to this equation. The $\gcd$ equation of $v$ can also be written
in the form 
\[
\det\left(\left[\begin{smallmatrix}b & x\\
-a & y
\end{smallmatrix}\right]\right)=1,
\]
which is equivalent to setting 
\begin{align*}
v^{\perp} & =\left(b,-a\right)^{\transpose}\\
w & =\left(x,y\right)^{\transpose}
\end{align*}
and requiring that 
\[
\left[\begin{array}{cc}
v^{\perp} & w\end{array}\right]\in\sl 2\left(\ZZ\left[\frac{1}{p}\right]\right).
\]
The possibilities for the second column $w$ are all the solutions
to the $\gcd$ equation of $v$; the matrices obtained from the different
possibilities are 
\[
\left\{ \left[\begin{array}{cc}
v^{\perp} & w_{0}+mv^{\perp}\end{array}\right]:m\in\ZZ\left[\frac{1}{p}\right]\right\} 
\]
where $w_{0}=\left(x_{0},y_{0}\right)^{\transpose}$ is some solution.
This set of matrices is an orbit for the group $\left\{ \left[\begin{smallmatrix}1 & \ZZ\left[\frac{1}{p}\right]\\
0 & 1
\end{smallmatrix}\right]\right\} =N_{p}\cap\sl 2\left(\ZZ\left[\frac{1}{p}\right]\right)$ acting by right multiplication:
\[
\left\{ \left[\begin{array}{cc}
v^{\perp} & w_{0}\end{array}\right]\left[\begin{array}{cc}
1 & m\\
0 & 1
\end{array}\right]:m\in\ZZ\left[\frac{1}{p}\right]\right\} .
\]

According to (\ref{eq: Iwasawa SL(2,R)}) and (\ref{eq: Bruhat-Iwasawa SL2(Qp)}),
when viewing this set of matrices (via the diagonal embedding) in
$\sl 2\left(\FF\right)$, it equals 
\[
\left\{ \left(k_{\infty}a_{\infty}\left[\begin{array}{cc}
1 & \frac{\left\langle w,\hat{v}\right\rangle }{\norm v}+m\\
0 & 1
\end{array}\right],k_{p}a_{p}\left[\begin{array}{cc}
1 & \frac{y\left(w\right)}{y\left(v\right)}+m\\
0 & 1
\end{array}\right]\right):m\in\ZZ\left[\frac{1}{p}\right]\right\} .
\]
Since $\mathcal{D}$ is a fundamental domain for $\ZZ\left[\frac{1}{p}\right]^{\diag}$
in $\FF$, there exists a unique $m$ for which $\left(\frac{\left\langle w,\hat{v}\right\rangle }{\norm v}+m,\frac{y\left(w\right)}{y\left(v\right)}+m\right)$
lies in $\mathcal{D}$. This $m=m\left(\mathcal{D}\right)$ determines
a unique solution $w_{v,\mathcal{D}}$ which in turn determines a
unique matrix 
\[
\ga_{v,D}:=\left[\begin{array}{cc}
v^{\perp} & w_{v,\mathcal{D}}\end{array}\right]\in\sl 2\left(\ZZ\left[\frac{1}{p}\right]\right).
\]
This establishes a one to one correspondence $v\leftrightarrow\ga_{v,\mathcal{D}}$
between $\ZZ\left[\frac{1}{p}\right]^{2}$ primitive vectors and matrices
in  $\sl 2\left(\ZZ\left[\frac{1}{p}\right]\right)\cap\leftgrp AN_{\mathcal{D}}$,
and proves part \ref{enu: 1 bijection} of the proposition. As for
part \ref{enu: bijection with parameters}: the fact that $\left(\dir v,\dirp v\right)\in\arc\times\arcp$
if and only if $\pi_{K}\left(\ga_{v}\right)\in\leftgrp_{\arc,\arcp}$
follows from (\ref{eq: Iwasawa SL(2,R)}) and the definition of $\leftgrp_{\arc,\arcp}$,
and the fact that $\norm v\leq R$ and $p^{t_{1}}\leq\normp v\leq p^{t_{2}}$
if and only if $\pi_{A}\left(\ga_{v}\right)\in A_{R,t_{1},t_{2}}$
follows from (\ref{eq: Bruhat-Iwasawa SL2(Qp)}) and the definition
of $A_{R,t_{1},t_{2}}$. 
\end{proof}

\section{Counting lattice points inside well-rounded sets\label{sec:Counting-lattice-points-1}}

Proposition \ref{prop: matrices representing  Z=00005B1/p=00005D primitive vectors }
gives a reformulation of Theorem \ref{thm: A ed of directions} in
the form of counting matrices of the lattice $\sl 2\left(\ZZ\left[\frac{1}{p}\right]\right)$
inside the $S$-arithmetic group $\sl 2\left(\FF\right)$. The next
step on the way to prove Theorem \ref{thm: A ed of directions} is
to solve this lattice point counting problem. This is the topic of
this section. The method we will apply was established in \cite{GN12},
and it relies on ergodic theory. The corpus of work on equidistribution
and counting lattice points by dynamical methods is rather vast; for
a short survey on the applied techniques in the case of lattices in
real algebraic Lie groups, we refer to \cite[p.7]{GN12}. As for the
techniques in (as well as an introduction to) the $S$-arithmetic
setting, we refer to the survey \cite{GUILL2014}. 

\subsection{Well-roundedness and counting lattice points}

Proposition \ref{prop: matrices representing  Z=00005B1/p=00005D primitive vectors }
allows us to translate the question on the number of primitive $\ZZ\left[\frac{1}{p}\right]^{2,+}$
vectors of real norm $\norm v\leq R$, $p$-adic norm $p^{t_{1}}\leq\normp v\leq p^{t_{2}}$
and directions $\left(\dir v,\dirp v\right)\in\arc\times\arcp\subseteq\sphere 1\times\sphere 1_{p}$,
to the problem of counting $\sl 2\left(\ZZ\left[\frac{1}{p}\right]\right)^{+}$
matrices in $\leftgrp_{\arc\times\arcp}A_{R,t_{1},t_{2}}N_{\mathcal{D}}$,
where $\mathcal{D}\subset\FF$ is a fundamental domain for $\ZZ\left[\frac{1}{p}\right]$.
From now on we fix the fundamental domain from Fact \ref{fact: fund dom for SL2(Z=00005B1/p=00005D)},
\[
\mathcal{D}:=\left(-\half,\half\right]\times\ZZ_{p}.
\]
We now describe a method to approach this counting problem.
\begin{notation}
\label{nota group}In what follows, $\algrp$ will denote the product
$\prod_{v\in S}\algrp_{v}\left(\KK_{v}\right)$ where $\KK$ is a
number field, $\KK_{v}$ is the localization of $\KK$ over a place
$v$, $\algrp_{v}$ is a simple algebraic group defined over $\KK$
and $S$ is a finite set of places that contains $\infty$. 
\end{notation}

\begin{defn}
\label{def: well--roundedness}Let $\algrp$ be as in Notation \ref{nota group},
$\mu$ a Borel measure on $\algrp$, and $\left\{ \nbhd{\e}{}\right\} _{\e>0}$
a family of identity neighborhoods in $G$. 
\end{defn}

\begin{enumerate}
\item For a measurable subset $\Gset\subset\algrp$, we define 
\[
\Gset^{+}\left(\e\right):=\nbhd{\e}{}\Gset\nbhd{\e}{}=\bigcup_{u,v\in\nbhd{\e}{}}u\,\Gset\,v,
\]
\[
\Gset^{-}\left(\e\right):=\bigcap_{u,v\in\nbhd{\e}{}}u\,\Gset\,v.
\]
\item The set $\Gset$ is \emph{Lipschitz well-rounded (LWR)} with (positive)
parameters $\left(\mathcal{C},\e_{0}\right)$ if for every $0<\e<\e_{0}$
\begin{equation}
\mu\left(\Gset^{+}\left(\e\right)\right)\leq\left(1+\mathcal{C}\e\right)\:\mu\left(\Gset^{-}\left(\e\right)\right).\label{eq:LWReq}
\end{equation}
\item A family $\left\{ \Gset_{R}\right\} _{R>0}\subset\algrp$ of measurable
domains is Lipschitz well-rounded with positive parameters $\left(\mathcal{C},R_{0},\e_{0}\right)$
if for every $0<\e<\e_{0}$ and $R>R_{0}$, the set $\Gset_{R}$ is
LWR with $\left(\mathcal{C},\e_{0}\right)$. 
\end{enumerate}
\begin{rem}
Indeed the definition of LWR depends on the choice of a family $\left\{ \nbhd{\e}{}\right\} _{\e>0}$
of identity neighborhoods; however, we disregard this fact since we
will only work with a specific family, cf. Definition \ref{def: p-adic id neighborhoods}. 
\end{rem}

The notion of well-roundedness of a family is by now standard (see
e.g. \cite{EM93}), but it is less common to define a well-rounded
\emph{set}; however, this notion will be useful for us, as the sets
under our consideration will project to a well-rounded family in the
real component, but to a well-rounded set in the $p$-adic one. 
\begin{thm}[{\cite[Theorems 1.9, 4.5, and Remark 1.10]{GN12}}]
\label{thm: GN Counting thm}Let $\algrp$ as in Notation \ref{nota group}
and with Haar measure $\mu$, and let $\gam<\algrp$ be a lattice.
Assume that $\left\{ \Gset_{R}\right\} \subset\algrp$ is a family
of finite-measure domains which satisfy $\mu\left(\Gset_{R}\right)\to\infty$
as $T\to\infty$. If the family $\left\{ \Gset_{R}\right\} $ is Lipschitz
well-rounded, then there exists a parameter $\errexp\left(\gam\right)\in\left(0,\frac{1}{2\left(1+\dim\algrp\right)}\right)$
such that for $R$ large enough and every $\delta>0$:
\[
\left|\#\left(\Gset_{R}\cap\gam\right)-\frac{\mu\left(\Gset_{R}\right)}{\mu\left(\algrp/\gam\right)}\right|\underset{\algrp,\Gamma,\delta}{\ll}\text{const}\cdot\mu\left(\Gset_{R}\right)^{1-\errexp\left(\gam\right)+\delta}
\]
as $T\to\infty$, where $\mu\left(\algrp/\gam\right)$ is the measure
of a fundamental domain of $\gam$ in $\algrp$.
\end{thm}

\begin{rem}[\textbf{The error exponent}]
\label{rem: err exponent}The parameter $\errexp\left(\gam\right)$
depends on estimates on the rate of decay of matrix coefficients of
automorphic representations of $\algrp$. For $\algrp=\sl 2\left(\RR\right)\times\sl 2\left(\QQ_{p}\right)$,
any bound toward the generalized Ramanujan conjecture, from Gelbart-Jacquet
\cite{GelJac78} to Kim-Sarnak \cite{KimSar03}, implies that these
coefficients are $L^{q+}$ for some $2<q\leq4$ (see e.g.\ \cite{COH01}).
These bounds give $\errexp_{p}=\errexp(\sl 2\left(\ZZ\left[\frac{1}{p}\right]\right))=\frac{1}{4\left(1+\dim\left(\algrp\right)\right)}=\frac{1}{28}$,
and only the full Ramanujan conjecture would give a better exponent,
namely $\errexp_{p}=\frac{1}{2\left(1+\dim\left(\algrp\right)\right)}=\frac{1}{14}$.
This exponent is obtained by a combination of Theorems 1.9, 4.5 and
Definition 3.1 in \cite{GN12}.
\end{rem}

According to Theorem \ref{thm: GN Counting thm}, the goal of counting
lattice points inside $\leftgrp_{\arc\times\arcp}A_{R,t_{1},t_{2}}N_{D}$
will be achieved by establishing that these sets are Lipschitz well-rounded
w.r.t.\ a certain choice of identity neighborhoods in $\sl 2\left(\RR\right)\times\sl 2\left(\QQ_{p}\right)$.
The LWR of these sets reduces to the LWR of their projections to
both the real and to the $p$-adic components; the LWR of the real
component in known (see more details in the proof of Theorem \ref{thm: A ed of directions}
below), and so it remains to verify the LWR of the projection to the
$p$-adic part. To this end, we will consider the following identity
neighborhoods inside $\sl 2\left(\QQ_{p}\right)$: 
\begin{defn}
\label{def: p-adic id neighborhoods}For any subgroup $\BIcomp_{p}$
of $G_{p}=\sl 2\left(\QQ_{p}\right)$ and a positive integer $N$
we set
\[
\nbhd{p^{-N}}{}=\ker\left\{ G\left(\ZZ_{p}\right)\to G\left(\ZZ_{p}/p^{N}\ZZ_{p}\right)\right\} =\idmat 2+p^{N}\Mat 2\left(\ZZ_{p}\right)\subset K_{p}
\]
and
\[
\nbhd{p^{-N}}{\BIcomp_{p}}=\nbhd{p^{-N}}{}\cap\BIcomp_{p}\subset K_{p}\cap H_{p}.
\]
\end{defn}

Let us now state explicitly the LWR property for the $p$-adic factor.
We note that the LWR property is rather strong here, since the Lipschitz
constant equals zero. 
\begin{prop}
\label{prop: p-adic well roundedness of sets}Consider the family
$\Gset_{t_{1},t_{2}}=\left(\leftgrp_{p}\right)_{\arcp}\left(A_{p}\right)_{t_{1},t_{2}}\left(N_{p}\right)_{\a+p^{\psi}\ZZ_{p}}$
of subsets in $\sl 2\left(\QQ_{p}\right)$ with:
\begin{enumerate}
\item $\arcp\subseteq\sphere 1_{p}$ a fixed $p$-adic arc;
\item $\psi\in\ZZ$;
\item $t_{1}$ and $t_{2}$ two real parameters satisfying $t_{1}\leq t_{2}$.
\end{enumerate}
Then the family $\left\{ \Gset_{t_{1},t_{2}}\right\} _{t_{0}<t_{1}\leq t_{2}}$
for an arbitrary $t_{0}\in\RR$ is Lipschitz well-rounded, with Lipschitz
constant zero. 
\end{prop}

The proof of Proposition \ref{prop: p-adic well roundedness of sets}
is quite technical, we postpone it to the end of the paper; see Section
\ref{sec: Proof of p-adic LWR}.\textcolor{red}{{} }Note that in the
following proof of Theorem \ref{thm: A ed of directions}, we will
only need the very special case where $t_{1}=t_{2}=0$. In this case
the family is indeed reduced to a single set!

We now have all the tools to prove Theorem \ref{thm: A ed of directions}. 
\begin{proof}[Proof of Theorem \ref{thm: A ed of directions}]
We first note that according to Lemma \ref{fact: primitive vectors over Z=00005B1/p=00005D},
\[
\ZZ_{\prim}^{2}=\ZZ\left[\frac{1}{p}\right]_{\prim}^{2}\cap\sphere 1_{p}.
\]
It then follows from Proposition \ref{prop: matrices representing  Z=00005B1/p=00005D primitive vectors }
that there is a bijection between the sets 
\[
\left\{ v\in\ZZ_{\prim}^{2}:\,\begin{matrix}\left(\dir v,v\right)\in\arc\times\arcp,\\
\norm v\leq R
\end{matrix}\right\} \longleftrightarrow\left\{ \ga\in\sl 2\left(\ZZ\left[\frac{1}{p}\right]\right)\cap\left(\leftgrp_{\arc\times\arcp}A_{R,0,0}N_{\left[-\frac{1}{2},\frac{1}{2}\right]\times\ZZ_{p}}\right)\right\} .
\]
The sets $\leftgrp_{\arc\times\arcp}A_{R,0,0}N_{\left[-\frac{1}{2},\frac{1}{2}\right]\times\ZZ_{p}}$
are a product of a real factor $\left(K_{\infty}\right)_{\arc}\left(A_{\infty}\right)_{R}\left(N\right)_{\left[-\frac{1}{2},\frac{1}{2}\right]}$
and a $p$-adic factor $\left(\leftgrp_{p}\right)_{\arcp}\left(A_{p}\right)_{0,0}\left(N_{p}\right)_{\ZZ_{p}}$;
the family of projections to real component is LWR by \cite[Theorem 1.1]{HN_Counting},
and the projection to the $p$-adic part is LWR according to Proposition
\ref{prop: p-adic well roundedness of sets}. Since a product of LWR
sets is LWR \cite[Corollary 4.3 and Remark 4.4]{HK_WellRoundedness},
it follows that the family $\left\{ \leftgrp_{\arc\times\arcp}A_{R,0,0}N_{\left[-\frac{1}{2},\frac{1}{2}\right]\times\ZZ_{p}}\right\} _{R>0}$
is LWR. In the notations of Theorem \ref{thm: GN Counting thm}, let
$\errexp_{p}=\errexp\left(\sl 2\left(\ZZ\left[\frac{1}{p}\right]\right)\right)=\frac{1}{4\left(1+\dim\left(G\right)\right)}$,
see Remark \ref{rem: err exponent}. By Theorem \ref{thm: GN Counting thm},
\begin{align*}
 & \#\left\{ \ga\in\sl 2\left(\ZZ\left[\frac{1}{p}\right]\right)\cap\left(\leftgrp_{\arc\times\arcp}A_{R,0,0}N_{\left[-\frac{1}{2},\frac{1}{2}\right]\times\ZZ_{p}}\right)\right\} \\
= & \frac{\mu_{\sl 2\left(\FF\right)}\left(\leftgrp_{\arc\times\arcp}A_{R,0,0}N_{\left[-\frac{1}{2},\frac{1}{2}\right]\times\ZZ_{p}}\right)}{\mu_{\sl 2\left(\FF\right)}\left(\sl 2\left(\FF\right)/\sl 2\left(\ZZ\left[\frac{1}{p}\right]\right)\right)}+O\left(\mu\left(\leftgrp_{\arc\times\arcp}A_{R,0,0}N_{\left[-\frac{1}{2},\frac{1}{2}\right]\times\ZZ_{p}}\right)^{1-\errexp_{p}+\delta}\right)\\
= & \frac{\mu_{\sl 2\left(\RR\right)}\left(\left(K_{\infty}\right)_{\arc}\left(A_{\infty}\right)_{R}\left(N_{\infty}\right)_{\left[-\frac{1}{2},\frac{1}{2}\right]}\right)}{\mu_{\sl 2\left(\RR\right)}\left(\sl 2\left(\RR\right)/\sl 2\left(\ZZ\right)\right)}\cdot\frac{\mu_{\sl 2\left(\QQ_{p}\right)}\left(\left(\leftgrp_{p}\right)_{\arcp}\left(A_{p}\right)_{0,0}\left(N_{p}\right)_{\ZZ_{p}}\right)}{\mu_{\sl 2\left(\QQ_{p}\right)}\left(\sl 2\left(\ZZ_{p}\right)\right)}\\
 & +O\left(\mu\left(\leftgrp_{\arc\times\arcp}A_{R,0,0}N_{\left[-\frac{1}{2},\frac{1}{2}\right]\times\ZZ_{p}}\right)^{1-\errexp_{p}+\delta}\right),
\end{align*}
where the last equality was deduced using Fact \ref{fact: fund dom for SL2(Z=00005B1/p=00005D)}.
The main term is a product of two factors, where the first one equals
\[
\frac{\leb(\arc)R^{2}}{\pi^{2}/6},
\]
by (\ref{eq: Iwasawa measure}). We turn to compute the second factor.
By Lemma \ref{fact: Haar measure on SL2(Q_p)} and Fact \ref{fact: p-adic half-sphere is Q},
\[
\mu_{\sl 2\left(\QQ_{p}\right)}\left(\left(K_{p}\right)_{\arcp}\left(A_{p}\right)_{0,0}\left(N_{p}\right)_{\ZZ_{p}}\right)=\measp^{2}\left(\arcp\right)\cdot1\cdot\measp\left(\ZZ_{p}\right)=\measp^{2}\left(\arcp\right).
\]
Furthermore, as noted in Fact \ref{fact: p-adic half-sphere is Q},
we have $\mu_{\sl 2\left(\QQ_{p}\right)}\left(\sl 2\left(\ZZ_{p}\right)\right)=\measp^{2}\left(\sphere{1,+}_{p}\right)$.
Then we may conclude: 
\[
\#\left\{ v\in\ZZ_{\prim}^{2}:\,\begin{matrix}\left(\dir v,v\right)\in\arc\times\arcp,\\
\norm v\leq R
\end{matrix}\right\} =\frac{6}{\pi^{2}}\frac{\leb(\arc)\measp^{2}\left(\arcp\right)}{\measp^{2}\left(\sphere{1,+}_{p}\right)}\cdot R^{2}+O\left(R^{2}\right)^{1-\errexp_{p}+\delta}.
\]
\end{proof}

\section{Equidistribution of Iwasawa components in the $S$-arithmetic case\label{sec:Equidistribution-of-Iwasawa}}

We have used a special case of Proposition \ref{prop: p-adic well roundedness of sets}
to prove Theorem \ref{thm: A ed of directions}. The full proposition
(together with Thorem \ref{thm: GN Counting thm}) implies a stronger
equidistribution result: the equidistribution of both $Q$ and $N$
components of the Bruhat-Iwasawa decomposition for an $S$-arithmetic
lattice, in product of $\sl 2$'s. Let us begin by introducing this
set up. 

So far we have only dealt with the group $\sl 2\left(\RR\right)\times\sl 2\left(\QQ_{p}\right)$,
meaning the case of one infinite place, and one finite place. But
in fact, we can allow any finite number of finite or infinite places,
and consider the group 
\[
G=\left(\sl 2\left(\RR\right)\right)^{n_{\infty}}\times\prod_{p\in S_{f}}\left(\sl 2\left(\QQ_{p}\right)\right)^{n_{p}},
\]
where $S_{f}$ is a finite set of primes. The notion of arithmetic
lattice, which generalizes $\sl 2\left(\ZZ\left[\frac{1}{p}\right]\right)<\sl 2\left(\RR\right)\times\sl 2\left(\QQ_{p}\right)$,
is described e.g. in \cite[Sec. 5.4]{PPR_92}. As noted in \cite[Rem. 4.6]{GN12},
the ergodic method of Gorodnik and Nevo applies to these lattices.
The analogous sets to $\leftgrp_{\arc\times\arcp}A_{R,t_{1},t_{2}}N_{\Psi_{\infty}\times\Psi_{p}}\subset\sl 2\left(\FF\right)$
inside $G$ are the following. For $1\leq i\leq n_{\infty}$, let
$\arc_{i}\subset\sphere 1$ be arcs on the unit circle, $\Nset_{i}\subset\RR$
intervals, and $R_{i}\geq1$ positive real numbers. Set 
\[
\underline{R}=(R_{1},\ldots,R_{n_{\infty}}),
\]
and consider 
\[
B_{\underline{R}}^{\text{Iw}}:=\prod_{i=1}^{n_{\infty}}\left(K_{\infty}\right)_{\arc_{i}}\left(A_{\infty}\right)_{R_{i}}\left(N_{\infty}\right)_{\Nset_{i}}\times\prod_{p\in S_{f}}\left(\leftgrp_{p}\right)_{\arcp}^{n_{p}}\left(A_{p}\right)_{t_{1}^{p},t_{2}^{p}}^{n_{p}}\left(N_{p}\right)_{\a_{p}+p^{\psi_{p}}\ZZ_{p}}^{n_{p}}
\]
where for every $p\in S_{f}$:
\begin{enumerate}
\item $\arcp\subseteq\sphere 1_{p}$ is a fixed $p$-adic arc,
\item $\psi_{p}\in\ZZ$,
\item $t_{1}^{p}$ and $t_{2}^{p}$ are two real parameters that satisfy
$t_{1}^{p}\leq t_{2}^{p}$ and are bounded from below, namely there
exist $t_{0}^{p}\in\RR$ such that $t_{0}^{p}<t_{1}^{p}\leq t_{2}^{p}$.
\end{enumerate}
Let $\mu$ denote a Haar measure on $G$.

\begin{thm}
\label{thm: Equidistribution Bruhat-Iwasawa}Let $\gam<G$ be an $S$-arithmetic
lattice, and $\norm{\cdot}$ any norm on $\RR^{n_{\infty}}$. Then,
for $\errexp=\frac{1}{4\left(1+3\left(n_{\infty}+\sum n_{p}\right)\right)}$,
 the following asymptotic formula holds for every $\dl>0$ as $\norm{\underline{{R}}}\to\infty$
:
\[
\#\left\{ \left(B_{\underline{R}}^{\text{Iw}}\right)\cap\gam^{+}\right\} =\frac{\mu\left(B_{\underline{R}}^{\text{Iw}}\right)}{\mu\left(G/\gam\right)}+O\left(\mu\left(B_{\underline{R}}^{\text{Iw}}\right)^{1-\errexp+\dl}\right)
\]
\begin{eqnarray*}
 & = & \frac{6}{\pi^{2}}\cdot\prod_{i=1}^{n_{\infty}}\left(\leb_{\sphere 1}(\arc_{i})\leb_{\RR}(\Nset_{i})\cdot R_{i}^{2}\right)\cdot\prod_{p\in S_{f}}\left(\frac{\measp^{2}\left(\arcp\right)\left(\sum_{t=t_{1}^{p}}^{t_{2}^{p}}p^{-2t}\right)\left(1-p^{-\psi_{p}}\right)}{1-\frac{1}{p}}\right)^{n_{p}}\\
 &  & +O\left((\prod_{i=1}^{n_{\infty}}R_{i}^{2})^{1-\errexp+\dl}\right).
\end{eqnarray*}
The implied constant depends on $\arc$, $\Nset$, and $\psi_{p}$,
$\arcp$, $t_{1}^{p}$, $t_{2}^{p}$ for every $p$. 
\end{thm}

The exponent $\tau$, as discussed in Remark \ref{rem: err exponent},
would be improved to $\errexp=\frac{1}{2\left(1+3\left(n_{\infty}+\sum n_{p}\right)\right)}$
with the full Ramanujan conjecture.
\begin{rem}
If $t_{2}^{p}$  grows to infinity, then the $O$-constant does not
depend on it. 
\end{rem}

\begin{proof}
We first note that the family $\left\{ B_{\underline{R}}^{\text{Iw}}\right\} $
is LWR, since it is a product of the projections $\left(K_{\infty}\right)_{\arc_{i}}\left(A_{\infty}\right)_{R_{i}}\left(N\right)_{I_{i}}$
to the real components, which are LWR according to \cite[Theorem 1.1]{HN_Counting},
and the projections $\left(\leftgrp_{p}\right)_{\arcp}\left(A_{p}\right)_{t_{1}^{p},t_{2}^{p}}\left(N_{p}\right)_{\a_{p}+p^{\psi_{p}}\ZZ_{p}}$
to the finite components, which are LWR according to Proposition \ref{prop: p-adic well roundedness of sets}.
Now the result follows from Theorem \ref{thm: GN Counting thm}, combined
with the fact that the $\mu_{\sl 2(\RR)}^{\times n_{\infty}}\times\prod_{p}\mu_{\sl 2(\QQ_{p})}$-volume
of $B_{\underline{R}}^{\text{Iw}}$ is the expression appearing in
the main term, according to \ref{eq: Iwasawa measure} and to Lemma
\ref{fact: Haar measure on SL2(Q_p)}. 
\end{proof}
\begin{rem}
\label{rem: S-arithmetic prim vectors}We note that when considering
the more general $S$-arithmetic setting as in Theorem \ref{thm: Equidistribution Bruhat-Iwasawa},
we obtain a generalization of Theorem \ref{thm: A ed of directions}
to joint equidistribution of several $p$-adic directions of primitive
vectors. For this, one should apply Theorem \ref{thm: Equidistribution Bruhat-Iwasawa}
to the lattice $\gam=\sl 2\left(\ZZ[\{\frac{1}{p}\}_{p\in S_{f}}]\right)$.
\end{rem}

\begin{rem}
\begin{enumerate}
\item Theorem \ref{thm: Equidistribution Bruhat-Iwasawa} is in fact an
equidistribution  results, since it is completely standard to pass
from a counting formulation to an equidistribution formulation; see,
for example, the proof of Theorem \ref{thm: A ed of directions}.
\item Equisitribution of Iwasawa components, and especially of the $N$-component,
has been considered in a number of papers, e.g. \cite{Good83,RR09,Tru13,Mar_Vino_15,HN_Counting}.
All of the above in rank one real Lie groups; in higher rank, we mention
\cite{HK_gcd}. We are unaware of equidistribution results of the
Iwasawa components of lattice elements in the $S$-arithmetic setting.
\item For more equidistribution and counting results in the $S$-arithmetic
or Adelic settings, we refer to \cite{COH01}, \cite{ELMV_09}, \cite{Benoist_Oh_12}
and \cite{GN12}. 
\end{enumerate}
\end{rem}

\section{Proof of Well-roundedness in $\protect\sl 2\left(\protect\QQ_{p}\right)$\label{sec: Proof of p-adic LWR}}

The goal of this final subsection is to prove Proposition \ref{prop: p-adic well roundedness of sets},
i.e. the well-roundedness of the sets $\left(\leftgrp_{p}\right)_{\arcp}\left(A_{p}\right)_{t_{1},t_{2}}\left(N_{p}\right)_{\a+p^{\psi}\ZZ_{p}}$
in $\sl 2\left(\QQ_{p}\right)$. The main step is a measurement of
how the Bruhat components are modified by a small left or right perturbation.
To state this proposition, we need an additional notation: for any
$g\in\sl 2\left(\QQ_{p}\right)$, we denote by $\|\Ad g\|_{op}$ the
operator norm of $\Ad g$ acting on $Mat_{2}(\QQ_{p})$. Note that
this operator norm takes values in $p^{\ZZ}$, as the max norm on
$\QQ_{p}^{2}$.
\begin{prop}[\textbf{Effective Bruhat-Iwasawa decomposition}]
\label{prop: Effective Iwasawa decomposition}Let $g=qan\in\sl 2^{+}\left(\QQ_{p}\right)$
 with $a=\left[\begin{smallmatrix}p^{-t} & 0\\
o & p^{t}
\end{smallmatrix}\right]$. Let $c\left(a,n\right)=\left\Vert \Ad n\right\Vert _{\text{op}}\max(p^{-t},1)\in p^{\ZZ}$.
The function $c$ is bounded when $n$ is restricted to a bounded
set and $t$ is bounded from below Moreover, we have 
\[
\nbhd{\e}{G_{p}}qan\nbhd{\e}{G_{p}}\in q\nbhd{c\left(a,n\right)\cdot\e}{\leftgrp_{p}}a\nbhd{c\left(a,n\right)\cdot\e}{N_{p}}n
\]
when $\e\in p^{-\NN}$ is small enough. 
\end{prop}

The proof of Proposition \ref{prop: Effective Iwasawa decomposition}
requires the following three Lemmas: 
\begin{lem}
\label{lem: Conjugation inflates by the norm of Ad}For $g\in\sl 2\left(\QQ_{p}\right)$,
\[
g\nbhd{p^{-N}}{}g^{-1}\subseteq\nbhd{\left\Vert \Ad g\right\Vert _{\text{op}}p^{-N}}{}
\]
where $\left\Vert \Ad g\right\Vert _{\text{op}}\in p^{\ZZ}$ is the
operator norm of conjugation by $g$. 
\end{lem}

\begin{proof}
Indeed, 
\begin{align*}
g\nbhd{p^{-N}}{}g^{-1} & \subseteq g\left(\idmat 2+p^{N}\Mat 2\left(\ZZ_{p}\right)\right)g^{-1}\\
 & =\left(\idmat 2+g\cdot p^{N}\Mat 2\left(\ZZ_{p}\right)\cdot g^{-1}\right)\\
 & \subseteq\left(\idmat 2+\left\Vert \Ad g\right\Vert _{\text{op}}p^{N}\Mat 2\left(\ZZ_{p}\right)\right)=\nbhd{\left\Vert \Ad g\right\Vert _{\text{op}}p^{-N}}{}.
\end{align*}
It is clear that $\left\Vert \Ad g\right\Vert _{\text{op}}$ is a
power of $p$, since, as an operator norm, it is the maximum of norms
of ($p$-adic) vectors. 
\end{proof}

\begin{lem}
\label{claim: |Ad(k)|=00003D1}For every $k_{p}\in K_{p}$, $\left\Vert \Ad{k_{p}}\right\Vert _{\text{op}}=1$. 
\end{lem}

\begin{proof}
We know that $K_{p}\cdot\sphere 1_{p}=\sphere 1_{p}$, and that this
action preserves the $p$-adic norm. For $\mtx\in\Mat 2\left(\QQ_{p}\right)$,
we need to show that $\norm{\Ad{k_{p}}\left(T\right)}=\norm T$, where
the norm on $\Mat 2\left(\QQ_{p}\right)$ is the operator norm. Indeed
\[
\norm{\Ad{k_{p}}\left(T\right)}=\sup_{x\in\sphere 1_{p}}\normp{\Ad{k_{p}}\left(T\right)\cdot x}=\sup_{x\in\sphere 1_{p}}\normp{k_{p}^{-1}Tk_{p}\cdot x}=\sup_{y\in\sphere 1_{p}}\normp{Ty}=\norm T.
\]
Then $\Ad{k_{p}}$ is norm preserving, and therefore has operator
norm $1$. 
\end{proof}
\begin{lem}
\label{lem: inclusion of e-balls}For any $N\geq1$,
\[
\nbhd{p^{-N}}{G_{p}}=\nbhd{p^{-N}}{\leftgrp_{p}}\nbhd{p^{-N}}{N_{p}},
\]
\[
\nbhd{p^{-N}}{\leftgrp_{p}}=\nbhd{p^{-N}}{N_{p}^{-}}\nbhd{p^{-N}}{M_{p}}.
\]
\end{lem}

\begin{proof}
We note that when $N\geq1$
\[
\nbhd{p^{-N}}{\leftgrp_{p}}=\left\{ \left(\begin{matrix}1+p^{N}a & 0\\
p^{N}b & \left(1+p^{N}a\right)^{-1}
\end{matrix}\right):a,b\in\ZZ_{p}\right\} ,
\]
\[
\nbhd{p^{-N}}{N_{p}}=\left\{ \left(\begin{matrix}1 & p^{N}\ZZ_{p}\\
0 & 1
\end{matrix}\right)\right\} 
\]
and 
\[
\nbhd{p^{-N}}{M_{p}}=\left\{ \left(\begin{matrix}1+p^{N}\a & 0\\
0 & \left(1+p^{N}\a\right)^{-1}
\end{matrix}\right):\a\in\ZZ_{p}\right\} .
\]
The inclusions $\supseteq$ in the statement of the lemma are trivial.
For the opposite direction, observe that
\[
\nbhd{p^{-N}}{G_{p}}\ni\left(\begin{matrix}1+p^{N}a & p^{N}c\\
p^{N}d & 1+p^{N}b
\end{matrix}\right)=\left(\begin{matrix}1+p^{N}x & 0\\
p^{N}y & \left(1+p^{N}x\right)^{-1}
\end{matrix}\right)\left(\begin{matrix}1 & p^{N}z\\
0 & 1
\end{matrix}\right)
\]
for 
\[
\begin{matrix}x=a, & y=d, & z=\frac{c}{1+p^{N}a}\end{matrix},
\]
as $\left(1+p^{N}x\right)^{-1}=1-\frac{P^{N}x}{1+p^{N}x}$ and the
determinant of the left-hand matrix is $1$. Similarly, 
\[
\nbhd{p^{-N}}{\leftgrp_{p}}\ni\left(\begin{matrix}1+p^{N}a & 0\\
p^{N}b & \left(1+p^{N}a\right)^{-1}
\end{matrix}\right)=\left(\begin{matrix}1 & 0\\
p^{N}z & 1
\end{matrix}\right)\left(\begin{matrix}1+p^{N}x & 0\\
0 & \left(1+p^{N}x\right)^{-1}
\end{matrix}\right)
\]
for 
\[
\begin{matrix}x=a, & z=\frac{b}{1+p^{N}a}\end{matrix}.\qedhere
\]
\end{proof}
We now turn to prove Proposition \ref{prop: Effective Iwasawa decomposition}. 
\begin{proof}[Proof of Proposition \ref{prop: Effective Iwasawa decomposition}]
Let $g=qan$ with $q\in\leftgrp$, $a=\left(\begin{smallmatrix}p^{-t} & 0\\
o & p^{t}
\end{smallmatrix}\right)\in A_{p}$ and $n=\left(\begin{smallmatrix}1 & x\\
0 & 1
\end{smallmatrix}\right)\in N_{p}$. We will use the fact that 
\begin{equation}
\begin{array}{ccccccc}
a^{-1}n_{x}a & = & \left(\begin{smallmatrix}p^{t} & 0\\
o & p^{-t}
\end{smallmatrix}\right)\left(\begin{smallmatrix}1 & x\\
0 & 1
\end{smallmatrix}\right)\left(\begin{smallmatrix}p^{-t} & 0\\
o & p^{t}
\end{smallmatrix}\right) & = & \left(\begin{smallmatrix}1 & xp^{2t}\\
0 & 1
\end{smallmatrix}\right) & = & n_{xp^{2t}}\\
\\
an_{y}^{-}a^{-1} & = & \left(\begin{smallmatrix}p^{-t} & 0\\
o & p^{t}
\end{smallmatrix}\right)\left(\begin{smallmatrix}1 & 0\\
y & 1
\end{smallmatrix}\right)\left(\begin{smallmatrix}p^{t} & 0\\
o & p^{-t}
\end{smallmatrix}\right) & = & \left(\begin{smallmatrix}1 & 0\\
yp^{2t} & 1
\end{smallmatrix}\right) & = & n_{yp^{2t}}^{-}
\end{array}\label{eq: conjugation by A}
\end{equation}
where $n_{x}=\left(\begin{smallmatrix}1 & x\\
0 & 1
\end{smallmatrix}\right)\in N_{p}$, $n_{y}^{-}=\left(\begin{smallmatrix}1 & 0\\
y & 1
\end{smallmatrix}\right)\in N_{p}^{-}$ and $a=\left(\begin{smallmatrix}p^{-t} & 0\\
o & p^{t}
\end{smallmatrix}\right)$.

\paragraph*{Step 1: Left perturbations. }

Set $\e=p^{-N}$. Since $\left\Vert \Ad q\right\Vert _{\text{op}}=1$
(Lemma \ref{claim: |Ad(k)|=00003D1}), then by Lemmas \ref{lem: Conjugation inflates by the norm of Ad}
and \ref{lem: inclusion of e-balls} we have that,
\begin{align*}
\nbhd{\e}{G_{p}}qan & =q\left(q^{-1}\nbhd{\e}{G_{p}}q\right)an\subseteq q\nbhd{\e}{G_{p}}an=q\nbhd{\e}{\leftgrp_{p}}\nbhd{\e}{N_{p}}an.
\end{align*}
According to (\ref{eq: conjugation by A}),
\[
=q\nbhd{\e}{\leftgrp_{p}}a\cdot a^{-1}\nbhd{\e}{N_{p}}an=q\nbhd{\e}{\leftgrp_{p}}\cdot a\cdot\nbhd{p^{-2t}\e}{N_{p}}n.
\]

\paragraph*{Step 2: Right perturbations. }

By letting $C\left(n\right)=\left\Vert \Ad n\right\Vert _{\text{op}}\in p^{\ZZ}$,
we have according to Lemma \ref{lem: Conjugation inflates by the norm of Ad}
that 
\[
qan\nbhd{\e}{G_{p}}=qa\left(n\nbhd{\e}{G_{p}}n^{-1}\right)n\subseteq qa\nbhd{C\left(n\right)\e}{G_{p}}n;
\]
By Lemma \ref{lem: inclusion of e-balls} and to (\ref{eq: conjugation by A}),
\begin{align*}
 & \subseteq qa\cdot\nbhd{C\left(n\right)\e}{N_{p}^{-}}\nbhd{C\left(n\right)\e}{M_{p}}\nbhd{C\left(n\right)\e}{N_{p}}\cdot n\\
 & =qa\cdot\nbhd{C\left(n\right)\e}{N_{p}^{-}}a^{-1}\cdot a\nbhd{C\left(n\right)\e}{M_{p}}\nbhd{C\left(n\right)\e}{N_{p}}\cdot n\\
 & =q\nbhd{p^{-2t}C\left(n\right)\e}{N_{p}^{-}}a\cdot\nbhd{C\left(n\right)\e}{M_{p}}\nbhd{C\left(n\right)\e}{N_{p}}\cdot n.
\end{align*}
Since $A$ and $M$ commute, 
\[
=q\nbhd{p^{-2t}C\left(n\right)\e}{N_{p}^{-}}\nbhd{C\left(n\right)\e}{M_{p}}\cdot a\nbhd{C\left(n\right)\e}{N_{p}}n.
\]
By letting $C\left(a\right)=\max\left\{ p^{-2t},1\right\} $ we have
\[
\subseteq q\nbhd{C\left(a\right)C\left(n\right)\e}{N_{p}^{-}}\nbhd{C\left(a\right)C\left(n\right)\e}{M_{p}}\cdot a\nbhd{C\left(n\right)\e}{N_{p}}n.
\]
and then by Lemma \ref{lem: inclusion of e-balls}
\[
\subseteq q\nbhd{C\left(a\right)C\left(n\right)\e}{\leftgrp_{p}}\nbhd{C\left(a\right)C\left(n\right)\e}{N_{p}}a\nbhd{C\left(n\right)\e}{N_{p}}n.
\]
Finally, by (\ref{eq: conjugation by A}) 
\begin{align*}
=q\nbhd{C\left(a\right)C\left(n\right)\e}{\leftgrp_{p}}a\cdot a^{-1}\nbhd{C\left(a\right)C\left(n\right)\e}{N_{p}}a\nbhd{C\left(n\right)\e}{N_{p}}n & =q\nbhd{C\left(a\right)C\left(n\right)\e}{\leftgrp_{p}}a\nbhd{p^{-2k}C\left(a\right)C\left(n\right)\e}{N_{p}}\nbhd{C\left(n\right)\e}{N_{p}}n\\
 & \subseteq q\nbhd{C\left(a\right)^{2}C\left(n\right)\e}{\leftgrp_{p}}a\nbhd{C\left(a\right)^{2}C\left(n\right)\e}{N_{p}}n.
\end{align*}
Combining the effect of both left and right perturbations, we obtain
that 
\[
\nbhd{\e}{G_{p}}qan\nbhd{\e}{G_{p}}\in q\nbhd{c\left(a,n\right)\cdot\e}{\leftgrp_{p}}a\nbhd{c\left(a,n\right)\cdot\e}{N_{p}}n
\]
where $c\left(a,n\right)=C\left(a\right)^{2}C\left(n\right)$ is a
power of $p$ (since $C\left(a\right)$ and $C\left(n\right)$ are)
that is bounded when $t$ is bounded from below and $n$ is restricted
to a bounded set. Require that $\e<c\left(a,n\right)^{-1}$ to obtain
that $c\left(a,n\right)\e\in p^{-\NN}$. 
\end{proof}
We can now prove Proposition \ref{prop: p-adic well roundedness of sets}.
The proof essentially relies on the ultrametric nature of $\QQ_{p}$:
a small enough perturbation of a ball is the ball itself. The first
claim (\ref{eq: epsilon balls swallowed}) is the translation of this
phenomenon in our setting.
\begin{proof}[Proof of Proposition \ref{prop: p-adic well roundedness of sets}]
We first claim that for $N\geq0$ large enough and $c\e\leq p^{-N}$
we have: 
\begin{equation}
\begin{cases}
\nbhd{c\e}{N_{p}}\cdot\left(N_{p}\right)_{\a+p^{\psi}\ZZ_{p}}\subseteq\left(N_{p}\right)_{\a+p^{\psi}\ZZ_{p}} & \text{ and }\left(N_{p}\right)_{\a+p^{\psi}\ZZ_{p}}\cdot\nbhd{c\e}{N_{p}}\subseteq\left(N_{p}\right)_{\a+p^{\psi}\ZZ_{p}}\\
\nbhd{c\e}{\leftgrp_{p}}\left(\leftgrp_{p}\right)_{\arcp}\subseteq\left(\leftgrp_{p}\right)_{\arcp} & \text{ and }\left(\leftgrp_{p}\right)_{\arcp}\cdot\nbhd{c\e}{\leftgrp_{p}}\subseteq\left(\leftgrp_{p}\right)_{\arcp}.
\end{cases}.\label{eq: epsilon balls swallowed}
\end{equation}
The inclusions in the first row are a trivial computation. For the
inclusions in the second row, write $\arcp=\arcp\left(\dirp v,p^{k}\right)$
where $k\geq0$ and $\dirp v\in\sphere{1,+}_{p}$. Let $N\geq0$ such
that $\nbhd{c\e}{}=\nbhd{p^{N}}{}$, and assume that $N\geq k$. Observe
that 

\[
\nbhd{c\e}{\leftgrp_{p}}=\nbhd{p^{N}}{\leftgrp_{p}}=\left\{ \left(\begin{matrix}1+p^{N}\ZZ_{p} & 0\\
p^{N}\ZZ_{p} & *
\end{matrix}\right)\right\} .
\]
By letting $\dirp v=\left(\begin{smallmatrix}u_{1}\\
p^{\ell}u_{2}
\end{smallmatrix}\right)\in\sphere{1,+}_{p}$ where $u_{1},u_{2}\in\ZZ_{p}^{\times}$ and $\ell\geq0$, then
\begin{eqnarray*}
\left(\leftgrp_{p}\right)_{\arcp\left(\dirp v,p^{k}\right)}=\left\{ \left(\begin{matrix}\dirp v+p^{k}\ZZ_{p}^{2} & \begin{smallmatrix}0\\*\end{smallmatrix}\end{matrix}\right)\right\}  & = & \left\{ \left(\begin{matrix}u_{1}+p^{k}\ZZ_{p} & 0\\
p^{\ell}u_{2}+p^{k}\ZZ_{p} & *
\end{matrix}\right)\right\} .
\end{eqnarray*}
Take $\left(\begin{smallmatrix}u_{1}+p^{k}\a & 0\\
p^{\ell}u_{2}+p^{k}\b & \left(u_{1}+p^{k}\a\right)^{-1}
\end{smallmatrix}\right)\in\left(\leftgrp_{p}\right)_{\arcp\left(\dirp v,p^{k}\right)}$ and $\left(\begin{smallmatrix}1+p^{N}\ga & 0\\
p^{N}\dl & \left(1+p^{N}\ga\right)^{-1}
\end{smallmatrix}\right)\in\nbhd{p^{-N}}{\leftgrp_{p}}$ (here $\a,\b,\ga\in\ZZ_{p}$). Now,
\begin{align*}
\left(\begin{smallmatrix}1+p^{N}\ga & 0\\
p^{N}\dl & \left(1+p^{N}\ga\right)^{-1}
\end{smallmatrix}\right)\cdot\left(\begin{smallmatrix}u_{1}+p^{k}\a & 0\\
p^{\ell}u_{2}+p^{k}\b & \left(u_{1}+p^{k}\a\right)^{-1}
\end{smallmatrix}\right) & \overset{N\geq k}{\in}\left\{ \left(\begin{smallmatrix}u_{1}+p^{k}\ZZ_{p} & 0\\
p^{\ell}u_{2}+p^{k}\ZZ_{p} & *
\end{smallmatrix}\right)\right\} \\
 & =\left(\leftgrp_{p}\right)_{\arcp\left(\dirp v,p^{k}\right)},\\
\\
\left(\begin{smallmatrix}u_{1}+p^{k}\a & 0\\
p^{\ell}u_{2}+p^{k}\b & \left(u_{1}+p^{k}\a\right)^{-1}
\end{smallmatrix}\right)\cdot\left(\begin{smallmatrix}1+p^{N}\ga & 0\\
p^{N}\dl & \left(1+p^{N}\ga\right)^{-1}
\end{smallmatrix}\right) & \overset{N\geq k}{\in}\left(\begin{smallmatrix}u_{1}+p^{k}\ZZ_{p} & 0\\
p^{\ell}u_{2}+p^{k}\ZZ_{p} & *
\end{smallmatrix}\right)\\
 & =\left(\leftgrp_{p}\right)_{\arcp\left(\dirp v,p^{k}\right)}.
\end{align*}

Having proved the inclusions in (\ref{eq: epsilon balls swallowed}),
the statement of Proposition \ref{prop: p-adic well roundedness of sets}
follows: according to Proposition \ref{prop: Effective Iwasawa decomposition},
when $g_{p}=qan\in G_{p}^{+}$ lies in $\leftgrp_{\arcp^{+}}A_{t_{1},t_{2}}N_{\a+p^{\psi}\ZZ_{p}}$,
then 
\[
\nbhd{\e}{G_{p}}qan\nbhd{\e}{G_{p}}\subseteq q\nbhd{c\e}{\leftgrp_{p}}a\nbhd{c\e}{N_{p}}n
\]
(where $c=c\left(a,n\right)$); but then according to (\ref{eq: epsilon balls swallowed}),
this is contained in $\leftgrp_{\arcp^{+}}A_{t_{1},t_{2}}N_{\a+p^{\psi}\ZZ_{p}}$.
Thus
\[
\left(\leftgrp_{\arcp^{+}}A_{t_{1},t_{2}}N_{\a+p^{\psi}\ZZ_{p}}\right)^{+\e}\subseteq\leftgrp_{\arcp^{+}}A_{t_{1},t_{2}}N_{\a+p^{\psi}\ZZ_{p}},
\]
and the opposite inclusion is obvious. Similarly, 
\[
\leftgrp_{\arcp^{+}}A_{t_{1},t_{2}}N_{\a+p^{\psi}\ZZ_{p}}\subseteq\left(\leftgrp_{\arcp^{+}}A_{t_{1},t_{2}}N_{\a+p^{\psi}\ZZ_{p}}\right)^{-\e}
\]
and the opposite inclusion is obvious. Then 
\[
\left(\leftgrp_{\arcp^{+}}A_{t_{1},t_{2}}N_{\a+p^{\psi}\ZZ_{p}}\right)^{-\e}=\leftgrp_{\arcp^{+}}A_{t_{1},t_{2}}N_{\a+p^{\psi}\ZZ_{p}}=\left(\leftgrp_{\arcp^{+}}A_{t_{1},t_{2}}N_{\a+p^{\psi}\ZZ_{p}}\right)^{+\e},
\]
meaning that the well-roundedness condition holds trivially. 
\end{proof}

\bibliographystyle{alpha}
\bibliography{bib_p-adic_gcd}

\begin{thebibliography}{ELMV09}

\bibitem[AES16a]{AES_16B}
M.~Aka, M.~Einsiedler, and U.~Shapira.
\newblock Integer points on spheres and their orthogonal grids.
\newblock {\em Journal of the London Mathematical Society}, 93(2):143--158,
  2016.

\bibitem[AES16b]{AES_16A}
M.~Aka, M.~Einsiedler, and U.~Shapira.
\newblock Integer points on spheres and their orthogonal lattices.
\newblock {\em Inventiones mathematicae}, 206(2):379--396, 2016.

\bibitem[BHC62]{Borel_HarishChandra}
A.~Borel and Harish-Chandra.
\newblock Arithmetic subgroups of algebraic groups.
\newblock {\em Annals of mathematics}, pages 485--535, 1962.

\bibitem[BO12]{Benoist_Oh_12}
Y.~Benoist and H.~Oh.
\newblock Effective equidistribution of s-integral points on symmetric
  varieties [{\'e}quidistribution effective des points s-entiers des
  vari{\'e}t{\'e}s sym{\'e}triques].
\newblock In {\em Annales de l'institut Fourier}, volume~62, pages 1889--1942,
  2012.

\bibitem[BS05]{BS05}
J.~Bryk and C.~E. Silva.
\newblock Measurable dynamics of simple p-adic polynomials.
\newblock {\em The American Mathematical Monthly}, 112(3):212--232, 2005.

\bibitem[COU01]{COH01}
L.~Clozel, H.~Oh, and E.~Ullmo.
\newblock Hecke operators and equidistribution of hecke points.
\newblock {\em Inventiones mathematicae}, 144(2):327--351, 2001.

\bibitem[Duk03]{Duke_03}
W.~Duke.
\newblock Rational points on the sphere.
\newblock In {\em Number Theory and Modular Forms}, pages 235--239. Springer,
  2003.

\bibitem[Duk07]{Duke_07}
W.~Duke.
\newblock An introduction to the {L}innik problems.
\newblock In {\em Equidistribution in number theory, an introduction}, pages
  197--216. Springer, 2007.

\bibitem[ELMV09]{ELMV_09}
M.~Einsiedler, E.~Lindenstrauss, P.~Michel, and A.~Venkatesh.
\newblock Distribution of periodic torus orbits on homogeneous spaces.
\newblock {\em Duke Mathematical Journal}, 148(1):119--174, 2009.

\bibitem[EM93]{EM93}
A.~Eskin and C.~McMullen.
\newblock Mixing, counting and equidistribution in {L}ie groups.
\newblock {\em Duke Mathematical Jurnal}, 71(1):181--209, 1993.

\bibitem[EMSS16]{EMSS_16}
M.~Einsiedler, S.~Mozes, N.~Sha, and U.~Shapira.
\newblock Equidistribution of primitive rational points on expanding
  horospheres.
\newblock {\em Compositio Mathematica}, 152(4):667--692, 2016.

\bibitem[ERW17]{ERW17}
M.~Einsiedler, R.~R{\"u}hr, and P.~Wirth.
\newblock Distribution of shapes of orthogonal lattices.
\newblock {\em Ergodic Theory and Dynamical Systems}, pages 1--77, 2017.

\bibitem[EW13]{EW13}
M.~Einsiedler and T.~Ward.
\newblock {\em Ergodic theory (with a view towards number theory)}.
\newblock Springer, 2013.

\bibitem[GI63]{GI63}
O.~Goldman and N.~Iwahori.
\newblock The space of $\mathfrak{p}$-adic norms.
\newblock {\em Acta Mathematica}, 109(1):137--177, 1963.

\bibitem[GJ78]{GelJac78}
S.~Gelbart and H.~Jacquet.
\newblock A relation between automorphic representations of ${\rm gl}(2)$ and
  ${\rm gl}(3)$.
\newblock {\em Annales scientifiques de l'\'Ecole Normale Sup\'erieure}, Ser.
  4, 11(4):471--542, 1978.

\bibitem[GN12]{GN12}
A.~Gorodnik and A.~Nevo.
\newblock Counting lattice points.
\newblock {\em Journal f{\"u}r die reine und angewandte Mathematik},
  2012(663):127--176, 2012.

\bibitem[Goo83]{Good83}
A.~Good.
\newblock On various means involving the {F}ourier coefficients of cusp forms.
\newblock {\em Mathematische Zeitschrift}, 183(1):95--129, 1983.

\bibitem[Gui14]{GUILL2014}
A.~Guilloux.
\newblock Equidistribution in $ s $-arithmetic and adelic spaces.
\newblock In {\em Annales de la Facult{\'e} des sciences de Toulouse:
  Math{\'e}matiques}, volume~23, pages 1023--1048, 2014.

\bibitem[HK19]{HK_gcd}
T.~Horesh and Y.~Karasik.
\newblock Equidistribution of primitive vectors, and the shortest solutions to
  their gcd equations.
\newblock {\em arXiv:1903.01560}, 2019.
\newblock arXiv preprint.

\bibitem[HK20]{HK_WellRoundedness}
T.~Horesh and Y.~Karasik.
\newblock A practical guide to well roundedness.
\newblock {\em arXiv:2011.12204}, 2020.
\newblock arXiv preprint.

\bibitem[HN16]{HN_Counting}
T.~Horesh and A.~Nevo.
\newblock Horospherical coordinates of lattice points in hyperbolic space:
  effective counting and equidistribution.
\newblock {\em arXiv:1612.08215}, 2016.
\newblock arXiv preprint.

\bibitem[Kim03]{KimSar03}
H.~H. Kim.
\newblock Functoriality for the exterior square of $\textrm{GL}_4$ and the
  symmetric fourth of $\textrm{GL}_2$. {W}ith an appendix by {D}.
  {R}amakrishnan, and an appendix co-authored by {P}. {S}arnak.
\newblock {\em Journal of the American Mathematical Society}, 16(1):139--183,
  2003.

\bibitem[Lin68]{Linnik_68}
Y.~V. Linnik.
\newblock {\em Ergodic Properties of Algebraic Fields}, volume~45 of {\em
  Ergebnisse der Mathematik und ihrer Grenzgebiete}.
\newblock Springer-Verlag Berlin Heidelberg, 1968.

\bibitem[Mar10]{Marklof_10}
J.~Marklof.
\newblock The asymptotic distribution of frobenius numbers.
\newblock {\em Inventiones mathematicae}, 181(1):179--207, 2010.

\bibitem[MV15]{Mar_Vino_15}
J.~Marklof and I.~Vinogradov.
\newblock Directions in hyperbolic lattices.
\newblock {\em Journal f{\"u}r die reine und angewandte Mathematik (Crelles
  Journal)}, 2015.

\bibitem[PR92]{PPR_92}
V.P. Platonov and A.S. Rapinchuk.
\newblock Algebraic groups and number theory.
\newblock {\em Russian Mathematical Surveys}, 47(2):133, 1992.

\bibitem[RR09]{RR09}
M.~Risager and Z.~Rudnick.
\newblock On the statistics of the minimal solution of a linear diophantine
  equation and uniform distribution of the real part of orbits in hyperbolic
  spaces.
\newblock {\em Contemporary Mathematics}, 484:187--194, 2009.

\bibitem[Sch98]{Schmidt_98}
W.~M. Schmidt.
\newblock The distribution of sub-lattices of {$Z^m$}.
\newblock {\em Monatshefte f{\"u}r Mathematik}, 125:37--81, 1998.

\bibitem[Tru13]{Tru13}
J.L. Truelsen.
\newblock Effective equidistribution of the real part of orbits on hyperbolic
  surfaces.
\newblock In {\em Proceedings of the American Mathematical Society}, volume
  141(2), pages 505--514, 2013.

\end{thebibliography}

\end{document}